\documentclass[11pt,a4paper,reqno]{amsart}

\usepackage[margin=1.0in]{geometry}
\usepackage[T1]{fontenc}
\usepackage[utf8]{inputenc}
\usepackage{lmodern}
\usepackage{amsmath,amsthm,amssymb,amsxtra,mathtools,bm,mathrsfs}
\usepackage{enumerate,enumitem}
\usepackage{booktabs,array,longtable}
\usepackage{graphicx}
\usepackage{xcolor}
\usepackage{cite}
\usepackage[colorlinks=true,citecolor=blue,linkcolor=black,urlcolor=blue]{hyperref}
\usepackage[title]{appendix}
\allowdisplaybreaks
\numberwithin{equation}{section}

\newtheorem{theorem}{Theorem}[section]

\newtheorem{lemma}[theorem]{Lemma}

\newtheorem{remark}[theorem]{Remark}
\theoremstyle{definition}

\newcommand{\Sph}{\mathbb S^2}
\newcommand{\R}{\mathbb R}
\newcommand{\C}{\mathbb C}
\newcommand{\V}{\mathbb V}

\newcommand{\Q}{\mathcal Q}
\newcommand{\Lhyp}{\mathcal L}

\newcommand{\dd}{\,\mathrm d}
\newcommand{\gradS}{\nabla_{\Sph}}

\newcommand{\LapS}{\Delta_{\Sph}}

\newcommand{\norm}[1]{\left\lVert #1\right\rVert}
\newcommand{\abs}[1]{\left\lvert #1\right\rvert}

\newcommand{\tr}{\operatorname{tr}}

\newcommand{\spanop}{\operatorname{span}}

\def\d{{\mathrm d}}

\def\ud{\underline{D}}

\begin{document}

\title[]{\parbox[b]{\linewidth}{\centering A spherical harmonic pseudo-spectral approach to mean curvature flow of surfaces with spherical topology}}

\author[]{Genming Bai}

\subjclass[2020]{65M70, 65N35, 65D32, 53E10, 53A04, 35K93}
\keywords{spherical harmonics, pseudo-spectral method, hyperinterpolation, quadrature, mean curvature flow, convergence, stability}

\begin{abstract}
	We propose and analyze a spherical harmonic pseudo-spectral method for the mean curvature flow of closed surfaces with spherical topology. The evolving surface is represented by a global parametrization over the unit sphere, and the continuous weak formulation underlying Dziuk’s method [G. Dziuk, Numer. Math., 1991] is discretized in a finite-dimensional space of spherical harmonics. By explicitly taking into account quadrature errors on the autonomously evolving numerical surface, we prove exponential convergence of the position error under the assumption that the initial global parametrization is analytic. The convergence analysis developed herein is general and could apply to other moving-domain and geometric evolution problems. Numerical experiments confirm the theoretical result.
\end{abstract}

\maketitle
\vspace{-1em}

\section{Introduction}

Let \(\Gamma(t)\subset \mathbb R^3\) be a smooth closed evolving surface.  The
mean curvature flow of \(\Gamma(t)\) can be written as
\begin{align}\label{eq:MCF}
	(\partial_t X) \circ X^{-1} = \Delta_{\Gamma(t)}{\rm id} = -Hn,
\end{align}
where \(X(t)\) is a global parametrization of $\Gamma(t)$, \(n(t)\) is
the outward unit normal, and \(H(t)\) is the scalar mean curvature on $\Gamma(t)$ with the sign
convention fixed by \eqref{eq:MCF}.  The weak formulation is
\begin{align}\label{eq:dziuk-weak}
	\int_{\Gamma(t)} (\partial_t X)\circ X^{-1} \cdot \varphi
	+
	\int_{\Gamma(t)} \nabla_{\Gamma(t)}{\rm id}\cdot\nabla_{\Gamma(t)}\varphi
	=0 .
\end{align}
Mean curvature flow \eqref{eq:MCF} and its analytic properties have been studied extensively in
\cite{CY2007,EH1989,EH1991,Eck12,Man11,Huisken1990,Ecker2001,White2005,Huisken1984}.
Because \eqref{eq:MCF} prescribes only the normal motion of the surface, it retains the degrees of freedom to choose a tangential velocity, or a tangential gauge.  The corresponding numerical analysis depends
strongly on the choice of tangential gauge.  For the numerical analyses in the graph gauge, we refer to
\cite{DD2006,LN2005,BMN2004,DD1995,DD2000}; for DeTurck-type gauges, see
\cite{EF2017,Fritz_thesis}. We note that the well-known Barrett--Garcke--N\"urnberg (BGN) method \cite{BGN2007JCP,BGN2008JCP} can be identified as the endpoint case of the DeTurck-gauged family of methods proposed by Elliott and Fritz in \cite{EF2017}; recent convergence analyses of BGN-type methods can be found in \cite{BL2025,BGV2026}.
The normal gauge, where the velocity is purely normal, is more difficult to analyze.  In contrast to graph and DeTurck gauges, whose evolution systems are quasilinear
and strictly parabolic, the normal-gauge formulation is degenerate and lacks
full parabolicity; see, for instance, \cite[Eq. (2.1)]{CY2007}.  A recent
projection-error framework \cite{BL2024} recovers the discrete full \(H^1\) parabolicity, leading to a convergence proof; also see \cite{BL22A,Li21}. An alternative approach to the normal gauge is to discretize Huisken's weak formulation
\cite{Huisken1984}, originally introduced by Huisken in the study of blow-up singularities
of convex surfaces; related convergence analyses can be found in
\cite{KLL19,KLL-Willmore}.

Spectral and pseudo-spectral methods on flat domains are well-developed; see the textbooks
\cite{CHQZ1988book,STW2011book,Trefethen2000book} for a thorough survey. On stationary surfaces, approximation theory has been studied extensively in
\cite{DX2013book,AH2012book,Xu2014}.  On the algorithmic side, Sloan introduced the concept of
hyperinterpolation in \cite{Sloan1995,SW2000} -- a computable discrete analogue of
spectral projection.
This idea has subsequently been used in the convergence proofs of pseudo-spectral methods for
second-kind integral equations \cite{GS2002} and for the Navier--Stokes
equations on the unit sphere \cite{GLS2011}.  

Note that most available rigorous analyses for mean curvature flow use parametric finite
element methods; see \cite{DDE2005,DE2013,BGN2020} for an overview. These methods
typically yield algebraic convergence rates.  By contrast, pseudo-spectral methods are known to be spectrally accurate, computationally efficient,
and scalable for geometric flows
\cite{VRBZ11,VGBZ09,VGZB09,ST2018,ZB2025,FBS2023}, but rigorous error analysis for pseudo-spectral methods on moving domains remains open.  This gap motivates the present work.

The purpose of this paper is to develop a systematic numerical analysis for a spherical harmonic pseudo-spectral approximation of mean curvature flow in the
normal gauge, i.e., \eqref{eq:MCF}. The techniques developed in this paper are not restricted to mean curvature flow and could be extended to more general moving-domain problems. Instead of using a triangulated surface and local finite element spaces, we assume that the exact surface admits a global parametrization
\(X(t):\Sph\rightarrow\Gamma(t)\).  The numerical surface is represented by a
vector-valued spherical harmonic polynomial.  We consider the following Dziuk-type semi-discrete Galerkin method: Given a suitable
initial approximation \(X_N(0)\in[\V_N]^3\), we seek
\(X_N(t)\in[\V_N]^3\), with \(\Gamma_N(t)=X_N(t)(\Sph)\), such that
\begin{equation*}
	\mathscr M_{\Gamma_N,\Q_R}\big(\partial_tX_N,\phi_N\big)
	+\mathscr A_{\Gamma_N,\Q_R}\big(X_N,\phi_N\big)
	=0,\quad\forall\phi_N\in[\V_N]^3 .
\end{equation*}

Here \(\Q_R\) is a positive quadrature rule with exactness order \(R\),
\(\mathscr M_{\Gamma_N,\Q_R}\) and \(\mathscr A_{\Gamma_N,\Q_R}\) are the
discrete mass and stiffness forms, and
\[
	\V_N=\operatorname{span}\{Y_{\ell m}:0\le \ell\le N,\ -\ell\le m\le \ell\}
\]
is the space of spherical harmonics on $\Sph$.  Under suitable regularity assumptions, the
main theorem (Theorem~\ref{thm:main}) shows exponential convergence of
\(X_N(t)\) to \(X(t)\), i.e.,
\begin{align}
	\norm{X_N-X}_{L_t^\infty L_x^2([0,T]\times\Sph)}
	+
	\norm{\nabla_{\Sph} (X_N-X) \cdot n\circ X}_{L_t^2L_x^2([0,T]\times\Sph)}
	\leq C e^{-c N} . \notag
\end{align}
Since the weak formulation \eqref{eq:dziuk-weak} involves nonlinear geometric
quantities, the computation cannot be carried out entirely in the frequency
domain. This motivates the use of a pseudo-spectral method, in which the
nonlinear terms are evaluated on the physical domain with suitable quadrature rules.
In parametric finite element methods, quadrature error can be made higher order in the \(h\)-convergence, whereas in pseudo-spectral methods quadrature/aliasing error must be estimated explicitly in terms of the spectral degree \(N\).
 Unlike the parametric finite element approach, the
analysis here is carried out primarily on the fixed reference sphere \(\Sph\),
where integration by parts can be performed freely. 
The exponential convergence of the spectral approximation is also important for
stability: It allows the powers of \(N\) produced by inverse inequalities to be
absorbed into the exponential rate, which prevents possible order loss.

Because the normal-gauge formulation lacks tangential parabolicity, the error
terms must be handled carefully. In particular, they have to be controlled by
the full \(L^2\) norm together with only the normal component of the \(H^1\)
seminorm. A key step in the proof is that, after integration by parts, the
critical error terms in \eqref{eq:pdf-J2-geometric-difference} and
\eqref{eq:B-est1} manifest a directional structure. This structure allows them to
be controlled by the normal parabolicity in
Section~\ref{sec:norm-parab}.

Another key ingredient is to prove that the quadrature does not destroy numerical stability on the moving surface; see, in particular, the estimates \eqref{eq:pdf-J22} and \eqref{eq:pdf-K2} for the critical terms. Heuristically, this is because the geometric distortion of the surface produces a factor involving $\gradS e_N$ at each point $x\in\Sph$, where $e_N\in\V_N$ denotes the formal error. Fortunately, $\gradS e_N\in\V_{N+1}$ (see Section~\ref{sec:sph-arithm}), so, provided that the quadrature order is chosen sufficiently high, these geometric distortion terms are precisely captured and resolved by the quadrature rule, leading to the numerical stability.

The remainder of this paper is organized as follows. Section~2 reviews spherical harmonics, their product structure, and the de-aliasing degrees required for the mass and stiffness forms. Section~3 introduces the positive quadrature rule, hyperinterpolation, and spectral approximation estimates. Section~4 formulates the pulled-back pseudo-spectral Dziuk-type Galerkin method and states the main convergence theorem. Sections~5 and 6 are devoted to establishing the consistency and stability estimates, respectively, including the normal-parabolicity decomposition. In Section~7, we combine these consistency and stability bounds with a continuation argument to prove exponential convergence. Finally, Section~8 presents numerical experiments to confirm the predicted spectral convergence.

\section{Spherical harmonics and spectral spaces}

\subsection{Definition of spherical harmonics}

On the unit sphere $\Sph=\{x\in\R^3:|x|=1\}$, we define the spherical harmonics
\begin{equation*}
  Y_{\ell m}(\theta,\varphi)
  =c_{\ell m}P_\ell^{|m|}(\cos\theta)e^{im\varphi},
  \qquad \theta\in[0,\pi],\quad \varphi\in[0,2\pi),
\end{equation*}
where $P_\ell^{|m|}$ is the associated Legendre function and the normalization constant $c_{\ell m}$ is
chosen so that
\begin{equation*}
  \int_{\Sph}Y_{\ell m}\overline{Y_{\ell' m'}}
  =\delta_{\ell\ell'}\delta_{mm'}.
\end{equation*}
The spherical harmonics are eigenfunctions of the Laplace--Beltrami operator:
\begin{equation*}
  -\LapS Y_{\ell m}=\lambda_\ell Y_{\ell m},
  \qquad \lambda_\ell=\ell(\ell+1),
  \qquad -\ell\le m\le\ell.
\end{equation*}
We denote the space of spherical harmonics by
\begin{equation*}
  \V_N=\bigoplus_{\ell=0}^{N}\mathscr H_\ell
      =\spanop\{Y_{\ell m}:0\le\ell\le N,-\ell\le m\le\ell\} ,
\end{equation*}
with $\dim\V_N=(N+1)^2$.
For any $v\in L^2(\Sph)$, it admits the following spectral decomposition
\begin{equation*}
  v=\sum_{\ell=0}^{\infty}\sum_{m=-\ell}^{\ell}\widehat v_{\ell m}Y_{\ell m},
  \qquad \widehat v_{\ell m}=\int_{\Sph}v\overline{Y_{\ell m}}.
\end{equation*}
The $L^2$ projection onto $\V_N$ then coincides with the spectral truncation
\begin{equation*}
	\Pi_Nv=\sum_{\ell=0}^{N}\sum_{m=-\ell}^{\ell}
	\widehat v_{\ell m}Y_{\ell m}.
\end{equation*}

\subsection{Spherical-harmonic arithmetic}\label{sec:sph-arithm}

Products of spherical harmonics are expanded through Gaunt coefficients:
\begin{equation*}
	Y_{\ell_1m_1}Y_{\ell_2m_2}
	=\sum_{\ell=|\ell_1-\ell_2|}^{\ell_1+\ell_2}
	\sum_{m=-\ell}^{\ell}G(\ell_1,m_1;\ell_2,m_2;\ell,m)Y_{\ell m},
\end{equation*}
and, consequently,
\begin{equation*}
	p_{N_1}q_{N_2}\in\V_{N_1+N_2}
	\qquad p_{N_1}\in\V_{N_1},\ q_{N_2}\in\V_{N_2}.
\end{equation*}
For surface gradient, it holds that
\begin{align}
	\gradS U_N\in [\V_{N+1}]^3 \qquad U_N\in\V_N. \notag
\end{align}
Given $U_N,V_N\in\V_N$, for their diagonally contracted gradients, we have
\begin{equation*}
	\gradS U_N\cdot\gradS V_N
	=
	\frac{1}{2}
	\left[
	\Delta_{\Sph}(U_N V_N)
	- \Delta_{\Sph}U_N V_N
	- U_N \Delta_{\Sph} V_N
	\right]
	\in\V_{2N}
	,
\end{equation*}
and if $A\in\R^{3\times 3}$ is a constant matrix, not necessarily diagonal, then there is an additional spectral shift of 2,
\begin{equation*}
	A\gradS U_N\cdot\gradS V_N\in\V_{2N+2}.
\end{equation*}
Thus a quadrature rule must be exact to at least degree $2N+2$ to
mitigate the influence of the quadrature error of the stiffness term.

\subsection{Analytic functions on \texorpdfstring{$\Sph$}{S2}}

We regard $\Sph$ as a compact real analytic manifold with its standard analytic
atlas.  A function $f:\Sph\to\R$ is real analytic, denoted by
$f\in C^\omega(\Sph)$, if for every analytic coordinate chart
$\kappa:U\subset \Sph\to \kappa(U)\subset\R^2$, the local representative
$f\circ\kappa^{-1}$ is real analytic in the usual Euclidean sense; equivalently,
near every point it is represented by a convergent power series in the local
coordinates \cite[Definition 2.7.4]{KP2002book}.  For vector-valued functions on $\Sph$, analyticity is understood componentwise.

On the compact analytic manifold $\Sph$, analyticity can be characterized by the decay rate of spectral coefficients \cite[Theorem]{Seeley1969}:
$f\in C^\omega(\Sph)$ if and only if the spectral coefficients decay
exponentially
\begin{equation}\label{eq:analytic-coefficient-decay}
  \left(\sum_{m=-\ell}^{\ell}\abs{\widehat f_{\ell m}}^2\right)^{1/2}
  \leq C e^{-c \ell},
  \qquad \ell=0,1,2,\ldots ,
\end{equation}
for some constants $C$ and $c$.
It follows from \eqref{eq:analytic-coefficient-decay} that each spherical
harmonic, and hence every finite linear combination of spherical harmonics, is
real analytic.

For mean curvature flow \eqref{eq:MCF}, the global parametrization has an additional tangential-gauge freedom, and therefore the parametrization itself is not automatically smoothed by the geometric equation.
After fixing the gauge by a DeTurck-type reparametrization
\cite{DeTurck1983, CY2007}, the evolution system becomes strictly parabolic and quasilinear.
Classical parabolic regularity theory then yields spatial
analyticity of each time slice; see \cite[Remark 1.5.3]{Man11}. Pulling the
solution back by the corresponding DeTurck diffeomorphisms shows that, for
analytic initial data, \eqref{eq:MCF} preserves analytic
regularity on the time interval of existence.
\begin{theorem}[Persistence of analyticity]\label{thm:cont-reg}
  Let $X_0:\Sph\to\R^3$ be a real analytic global parametrization. Then there exists a classical solution  $X:[0,T]\times\Sph\to\R^3$ to \eqref{eq:MCF} with
  $X(0)=X_0$ such that
  \[
    X(t,\cdot)\in C^\omega(\Sph;\R^3),\qquad 0\le t\le T ,
  \]
  for some $T>0$.
\end{theorem}

\section{Quadrature and hyperinterpolation on \texorpdfstring{$\Sph$}{S2}}

\subsection{Quadrature rule}

Let $\Q_R$ be a generic quadrature rule on $\Sph$, taking the form
\begin{align}
	\Q_R(f) = \sum_{j=1}^{M_{\Q_R}} w_j f(x_j), \notag
\end{align}
where $w_j\in\R$ and $x_j\in\Sph$, $j=1,\ldots,M_{\Q_R}$, are the quadrature
weights and nodes.
Throughout this paper, we assume that
\begin{itemize}
	\item[(A)] The weights of $\Q_R$ are positive.
	\item[(B)] $\Q_R$ is exact for $\V_R$, i.e., $\Q_R(f)=\int_{\Sph} f$ for all $f\in\V_R.$
\end{itemize}
We define the discrete sesquilinear inner product as
\begin{align}
	(u, v)_{\Q_R} := \Q_R(u \overline{v}) . \notag
\end{align}
Under the order condition $R\geq 2N$, for every $P_N\in\V_N$, we have
\begin{equation*}
  \Q_R(|P_N|^2)=\norm{P_N}_{L^2(\Sph)}^2,
\end{equation*}
and the discrete H\"older inequality:
\begin{equation*}
	|\Q_R(f P_N)|
	\leq \Q_R(|f|^2)^{1/2}\Q_R(|P_N|^2)^{1/2}
	\leq C \norm{f}_{L^\infty(\Sph)} \norm{P_N}_{L^2(\Sph)}.
\end{equation*}
The $L^\infty$-stability and exactness of the quadrature rule imply
\begin{equation*}
	\left|\left(\Q_R - \int_{\Sph}\right)f \right|
	\leq C \inf_{p_R\in\V_R}\norm{f-p_R}_{L^\infty(\Sph)}.
\end{equation*}

\subsection{Hyperinterpolation}

The degree-$N$ hyperinterpolation operator associated with $\Q_R$ is defined as
\begin{equation}\label{eq:hyperinterp}
  \Lhyp_N f
  =\sum_{\ell=0}^{N}\sum_{m=-\ell}^{\ell}
     (f, {Y_{\ell m}})_{\Q_R}Y_{\ell m}.
\end{equation}
If $R\geq 2N$, it holds that
(cf. \cite[Lemma 4 and Lemma 5]{Sloan1995}):
\begin{align}
	\Lhyp_N^2 = \Lhyp_N,\qquad
	\Lhyp_N|_{\V_N} = {\rm id}_{\V_N}, \notag
\end{align}
i.e., $\Lhyp_N$ is a projection onto $\V_N$.  Moreover, one can verify that
\begin{align}
	(\Lhyp_N f, g)_{\Q_R}
	= (f, \Lhyp_N g)_{\Q_R}
	= (\Lhyp_N f, \Lhyp_N g)_{\Q_R}
	= (\Lhyp_N f, \Lhyp_N g),
	\qquad
	(\Lhyp_N f, g)
	=
	(f, \Pi_N g)_{\Q_R}. \notag
\end{align}
Thus, $\Lhyp_N$ is self-adjoint with respect to the discrete inner product
$(\cdot,\cdot)_{\Q_R}$.  In particular, choosing $g=1$ gives the reciprocal
relation associated with \eqref{eq:hyperinterp}:
\begin{align}
	\Q_R(f) = \int_{\Sph} \Lhyp_N f, \notag
\end{align}
which is reminiscent of the Newton--Cotes quadrature rules.


\subsection{Approximation properties}

Let $W^{s,p}(\Sph),s\geq0,p\in[1,\infty]$, be the Sobolev–Slobodeckij space on the sphere, and denote $H^s(\Sph) := W^{s,2}(\Sph)$. Using the spectral decomposition, we have the norm equivalence
\begin{equation*}
	\norm{v}_{H^s(\Sph)}^2
	\sim_s\sum_{\ell=0}^{\infty}\sum_{m=-\ell}^{\ell}
	(1+\ell(\ell+1))^{s} |\widehat v_{\ell m}|^2.
\end{equation*}
The spectral truncation operator has a natural bound
\begin{equation*}
	\norm{v-\Pi_Nv}_{H^t(\Sph)}
	\leq C_{s,t} N^{t-s}\norm{v}_{H^s(\Sph)}.
\end{equation*}
Combining the Bernstein inequality \cite[Eq. (4.21)]{LM2006}
and the Nikolskii inequality \cite[Eq. (1.3)]{DGT2020}, we have the following inverse inequality on $\V_N$
\begin{equation*}
	\norm{p_N}_{W^{s,q}(\Sph)}
	\leq C_{s,t} N^{s-t+2(1/p-1/q)}\norm{p_N}_{W^{t,p}(\Sph)},
\end{equation*}
for all $p_N\in\V_N$, $s\ge t$, and $1\le p\le q\le\infty$; for an asymptotically sharp, but more complicated bound, see \cite{DFT2016}.


With the best approximation error defined by
\begin{align}
	E_N(f)_{s,p} := \inf_{g\in\V_N} \norm{f-g}_{W^{s,p}(\Sph)},\quad 1\leq p\leq \infty, \notag
\end{align}
the following upper bound is proved in \cite[Theorem 2.6.3, Proposition 2.6.4, and Corollary 4.5.6]{DX2013book}
\begin{align}\label{eq:BAE}
	E_N(f)_{s,p} \leq C_{s,t} N^{-t+s} \norm{f}_{W^{t,p}(\Sph)},
\end{align}
for all $0\leq s\leq t$
and $1\leq p\leq \infty$, and furthermore the upper bound of \eqref{eq:BAE} can be realized by some $g_N^*\in\V_N$.
With this result, we are able to show the stability and approximation properties of $\Pi_N$ and $\Lhyp_N$ in the following lemma.
\begin{lemma}
		We have
	\begin{align}
			\| (1 - \Pi_N) f \|_{W^{s,p}(\Sph)}
			\leq
			C_{s,t} N^{-t+s+|1/2-1/p|}
			\| f \|_{W^{t, p}(\Sph)} , \notag
	\end{align}
	and
	\begin{align}
			\| (1 - \Lhyp_N) f \|_{W^{s,p}(\Sph)}
			\leq
			C_{s,t} N^{-t+s+\max\{1/2-1/p,0\}}
			\| f \|_{W^{t,\infty}(\Sph)}. \notag
	\end{align}
	In particular, choosing $s=t$ and applying the triangle inequality gives the
	high-order stability estimates
	\begin{align}
				\| \Pi_N f \|_{W^{t,p}(\Sph)}
				\leq
				C_{t} N^{|1/2-1/p|} \| f \|_{W^{t, p}(\Sph)} , \notag
	\end{align}
	and
	\begin{align}
			\| \Lhyp_N f \|_{W^{t,p}(\Sph)}
			\leq
			C_{t} N^{\max\{1/2-1/p, 0\}}
			\| f \|_{W^{t,\infty}(\Sph)}. \notag
	\end{align}
\end{lemma}
\begin{proof}
		By definition, $\Pi_N$ is naturally stable in $L^2(\Sph)$.  Its mapping properties in
		$L^1(\Sph)$ and $L^\infty(\Sph)$ are given in
		\cite[Theorem 2.4.1]{DX2013book}. Then, complex interpolation gives
	\begin{align}
			\| \Pi_N f \|_{L^p(\Sph)}
			\leq
			C N^{|1/2-1/p|} \| f \|_{L^p(\Sph)}. \notag
	\end{align}
		The stability estimate for $\Lhyp_N$ follows from
		\cite[Theorem 1]{Sloan1995}, provided that the associated quadrature rule
		satisfies the exactness and positivity assumptions stated above:
	\begin{align}
				\| \Lhyp_N f \|_{L^2(\Sph)}
				\leq
				C \| f \|_{L^\infty(\Sph)} , \notag
	\end{align}
	and, according to \cite[Theorem 5.5.2]{SW2000},
	\begin{align}
			\| \Lhyp_N f \|_{L^\infty(\Sph)}
			\leq
			C N^{1/2} \| f \|_{L^\infty(\Sph)}. \notag
	\end{align}
	Consequently, complex interpolation gives
	\begin{align}
			\| \Lhyp_N f \|_{L^p(\Sph)}
			\leq
			C N^{1/2-1/p} \| f \|_{L^\infty(\Sph)}, \notag
	\end{align}
	for all $2\leq p\leq \infty$.  For $1\leq p\leq 2$, we can use the trivial inequality
	\begin{align}
			\| \Lhyp_N f \|_{L^p(\Sph)}
			\leq
			C \| \Lhyp_N f \|_{L^2(\Sph)}
			\leq
			C \|  f \|_{L^\infty(\Sph)}. \notag
	\end{align}
	In summary, we obtain
	\begin{align}
		\| \Lhyp_N f \|_{L^p(\Sph)}
		\leq
		C N^{\max\{ 1/2-1/p, 0 \}} \| f \|_{L^\infty(\Sph)}. \notag
	\end{align}
	Using the mapping properties obtained above and the inverse inequality, we derive
	\begin{align}
		&\| (1 - \Pi_N) f \|_{W^{s,p}(\Sph)}
		=
		\| (1 - \Pi_N) (f - g_N^*) \|_{W^{s,p}(\Sph)}
		\notag\\
		&\leq
		C_s N^s\| \Pi_N (f - g_N^*) \|_{L^{p}(\Sph)}
		+
		\| f - g_N^* \|_{W^{s,p}(\Sph)}
		\notag\\
		&\leq
		C_s N^{s+|1/2-1/p|}\| f - g_N^* \|_{L^{p}(\Sph)}
		+
		\| f - g_N^* \|_{W^{s,p}(\Sph)}
		\notag\\
		&\leq
		C_{s,t} N^{-t+s+|1/2-1/p|} \|  f \|_{W^{t,p}(\Sph)}
		+
		C_{s,t} N^{-t+s} \|  f \|_{W^{t,p}(\Sph)}
		\notag\\
		&\leq
		C_{s,t} N^{-t+s+|1/2-1/p|} \| f \|_{W^{t, p}(\Sph)} . \notag
	\end{align}
	A verbatim argument also carries over to $\| (1 - \Lhyp_N) f \|_{W^{s,p}(\Sph)}$, so its proof is omitted here.
\end{proof}

\begin{lemma}[Analytic approximation]\label{lemma:anal-approx}
		The analytic approximation estimates for $\Pi_N$ and $\Lhyp_N$ read
	\begin{align}
		\| (1 - \Pi_N) f \|_{W^{s, p}(\Sph)}
		+
		\| (1 - \Lhyp_N) f \|_{W^{s, p}(\Sph)}
		\leq C_s e^{-c_s N} , \notag
	\end{align}
		where the constants $C_s$ and $c_s$ depend on $s$ and $f$.
\end{lemma}
\begin{proof}
		Combining Sobolev embedding and the exponential decay of the spectral coefficients, i.e., \eqref{eq:analytic-coefficient-decay}, we have
	\begin{align}
		&\| (1 - \Pi_N) f \|_{W^{s,p}(\Sph)}
		\leq
		C \| (1 - \Pi_N) f \|_{H^{s+2}(\Sph)}
		\notag\\
		&\leq C_s
		\bigg(\sum_{\ell=N+1}^\infty \sum_{m=-\ell}^{\ell}
		(1+\ell^2)^{s+2}  |{\widehat f_{\ell m}}|^2\bigg)^{1/2}
		\leq
		C_s
		\bigg(\sum_{\ell=N+1}^\infty
		(1+\ell^2)^{s+2} e^{-c\ell}\bigg)^{1/2}
		\leq C_s e^{-c_s N} . \notag
	\end{align}
		Similarly, using the $W^{s,\infty}\rightarrow W^{s,p}$ stability of $\Lhyp_N$, we obtain
	\begin{align}
		&\| (1 - \Lhyp_N) f \|_{W^{s,p}(\Sph)}
		=
		\| (1 - \Lhyp_N) (1 - \Pi_N) f \|_{W^{s,p}(\Sph)}
		\notag\\
		&
		\leq
		\| (1 - \Pi_N) f \|_{W^{s, p}(\Sph)}
		+
		C_s N^{1/2} \| (1 - \Pi_N) f \|_{W^{s,\infty}(\Sph)}
		\notag\\
		&
		\leq (C_s + C_s N^{1/2}) e^{-c_s N}
		\leq C_s e^{-c_s N} , \notag
	\end{align}
	where, in the last inequality, we absorb $N^{1/2}$ into the exponential after reducing the generic constant $c_s$ suitably.
\end{proof}

\section{Geometry notation and the main theorem}

\subsection{Geometry notation}

We consider a spectral Galerkin method in a finite-dimensional space of spherical harmonics. The numerical surface at time $t$,
denoted by $\Gamma_N(t)$, is parametrized by
$X_N(t):\Sph\to\Gamma_N(t)$, whose components belong to $\V_N$.  The exact surface
is denoted by $\Gamma(t)$ and is parametrized by
$X(t):\Sph\to\Gamma(t)$.  We also define the following spectral comparison surface for consistency analysis
\[
  X_{N,*}(t)=\Lhyp_N X(t),\qquad
  \Gamma_{N,*}(t)=X_{N,*}(t)(\Sph).
\]
We denote by $\d X(t):T\Sph\to T\Gamma(t)$ the differential of
$X(t):\Sph\to\Gamma(t)$ in the Riemannian sense.  
Note that $\d X(t)$ and $(\d X(t))^{-1}$ admit the following natural extensions to the ambient
Euclidean space $\R^3$ at any $x\in\Sph$,
\begin{align}
	\overline{\d X(t)}|_x v
	:=
	\begin{cases}
		\d X(t)|_x v, & \mbox{if $v\in T\Sph_x$},\\
		0, & \mbox{if $v\in N\Sph_x$},
	\end{cases} \notag
\end{align}
\begin{align}
	\overline{(\d X(t))^{-1}}|_x v
	:=
	\begin{cases}
		(\d X(t)|_x)^{-1} v,
		& \mbox{if $v\in T\Gamma(t)_{X(t,x)}$},\\
		0, & \mbox{if $v\in N\Gamma(t)_{X(t,x)}$} .
	\end{cases} \notag
\end{align}
These extended mappings can be naturally interpreted as $\R^{3\times 3}$
matrices at a given point $x\in\Sph$.  With this convention, their matrix
representations coincide with
$\nabla_{\Sph} X(t, x)$ and
$(\nabla_{\Sph} X(t,x))^\dagger$, respectively, where $\dagger$ denotes the
Moore--Penrose pseudo-inverse (cf. \cite[Chapter 17]{Roman2005book}).

We first define the scalar Jacobian, or the volume distortion factor,
\[
  J(t, x):=\det(\d X(t)|_x) = {\rm pdet}(\nabla_{\Sph} X(t, x)) > 0,
\]
where ${\rm pdet}$ denotes the pseudo-determinant, i.e., the product of the
non-zero singular values, and for simplicity we assume that the orientation is positive.  We also define the matrix-valued coefficient on $\Sph$ by
\begin{align}
	A(t, x)
	&:=\overline{(\d X(t))^{-1}}|_x
	\big(\overline{(\d X(t))^{-1}}|_x\big)^T J(t, x)
	\notag\\
	&=(\nabla_{\Sph} X(t,x))^\dagger
	\big((\nabla_{\Sph} X(t,x))^\dagger\big)^T J(t, x). \notag
\end{align}
The definitions imply that $J$ and $A$ depend smoothly on $t$ and
$\nabla_{\Sph}X$ in a neighborhood of the exact solution. Therefore, we can write
\begin{align}\label{eq:JA-bar}
	J(t,x) = \bar J(t, \nabla_{\Sph} X(t, x)), \qquad
	A(t,x) = \bar A(t, \nabla_{\Sph} X(t, x)).
\end{align}
Similar definitions apply to the comparison surface
$X_{N,*}(t):\Sph\rightarrow\Gamma_{N,*}(t)$:
\begin{equation*}
	J_{N,*}={\rm pdet}(\nabla_{\Sph} X_{N,*}(t, x)) ,
	\qquad
	A_{N,*}
	=J_{N,*}\,(\nabla_{\Sph} X_{N,*}(t,x))^\dagger
	\big((\nabla_{\Sph} X_{N,*}(t,x))^\dagger\big)^T,
\end{equation*}
and to the numerical surface $X_N(t):\Sph\rightarrow\Gamma_N(t)$:
\begin{equation*}
	J_N={\rm pdet}(\nabla_{\Sph} X_{N}(t, x)),
	\qquad
	A_N
	=J_N\,(\nabla_{\Sph} X_N(t,x))^\dagger
	\big((\nabla_{\Sph} X_N(t,x))^\dagger\big)^T.
\end{equation*}
	These coefficients are the pullback quantities that transfer integrals on the
	moving surfaces to the fixed reference sphere.  More precisely, let
	$(F,J_F,A_F)$ be one of
	\[
		(X(t),J(t),A(t)),\qquad
		(X_{N,*}(t),J_{N,*}(t),A_{N,*}(t)),\qquad
		(X_N(t),J_N(t),A_N(t)),
	\]
	and let $\Sigma_F:=F(\Sph)$.  For functions $u,v$ defined on $\Sigma_F$, write
	$\widetilde u:=u\circ F$ and $\widetilde v:=v\circ F$ for their pullbacks to
	$\Sph$.  Then
	\begin{align*}
		\int_{\Sigma_F}uv
		&=\int_{\Sph}J_F\,\widetilde u\,\widetilde v,
		\\
		\int_{\Sigma_F}\nabla_{\Sigma_F}u\cdot\nabla_{\Sigma_F}v
		&=\int_{\Sph}\big(A_F\gradS\widetilde u\big)
		\cdot\gradS\widetilde v .
	\end{align*}

As in parametric finite element methods
\cite{DE2013,DDE2005,BL2024,KLL19}, the Galerkin functions are represented by
finite coefficient vectors.  Thus any function $f\in\V_N$ can be identified
uniquely with its coefficient vector, in analogy with the node-value vector in
finite element methods.
To avoid cumbersome pushforward and pullback notation, we adopt the convention
used in \cite{BL2024}.  For any discrete parametrization
$F_N:\Sph\rightarrow \Sigma_N\subset\R^3$, define the pushforward space of spherical harmonics
\begin{align}
	(F_N)_*(\V_N) := \{ f_N\in C(\Sigma_N): f_N\circ F_N\in\V_N \}. \notag
\end{align}
This pushforward map $(F_N)_*$ is one-to-one, so $f_N$ and $f_N\circ F_N$
share the same coefficient vector
$\mathbf f_N\in\C^{\dim(\V_N)}$.
With this identification in mind, we can write
\begin{align}
	\int_{\Sph} f_N,\qquad \int_{\Gamma_N} f_N,\qquad \int_{\Gamma_{N,*}} f_N \notag
\end{align}
without ambiguity: Each integrand is the realization of the same coefficient
vector $\mathbf f_N$ on a different surface.  In the rest of the paper, we use
this identification whenever the underlying surface is clear from context.

\subsection{The scheme and the convergence theorem}

To discretize \eqref{eq:dziuk-weak}, we consider the following
Dziuk-type pseudo-spectral method:
\begin{equation}\label{eq:pdf-semi-scheme}
  \mathscr M_{\Gamma_N,\Q_R}\big(\partial_tX_N,\phi_N\big)
  +\mathscr A_{\Gamma_N,\Q_R}\big(X_N,\phi_N\big)
  =0,
\end{equation}
where $X_N,\phi_N\in[\V_N]^3$.  The discrete inner
product on $\Gamma_N$ is defined by pulling back to the reference domain $\Sph$:
\begin{equation*}
  \mathscr M_{\Gamma_N,\Q_R}(\partial_t X_N,\phi_N)
  :=\Q_R(\partial_t X_N\cdot\phi_N\,J_N) ,
\end{equation*}
and
\begin{equation*}
	\mathscr A_{\Gamma_N,\Q_R}(X_N,\phi_N)
	:=\Q_R \big(A_N \nabla_{\Sph} X_N\cdot
	\nabla_{\Sph}(\phi_N)\big)  .
\end{equation*}
Note that according to the identification described above, $X_N$ and $\phi_N$ on the left-hand side are realized as the push-forward spherical harmonics on $\Gamma_N$, while $X_N$ and $\phi_N$ on the right-hand side are realized as functions in $[\V_N]^3$. There is no ambiguity since the underlying domains are clear on both sides.

In the remainder of this paper, we shall prove the exponential convergence of $X_{N}(t)$ to $X(t)$.
\begin{theorem}\label{thm:main}
	Assume that the initial parametrization $X(0)$ of the exact mean curvature
	flow \eqref{eq:MCF} is analytic, and suppose that the quadrature rule $\Q_R$ satisfies the
	order condition $R\geq 3N$.  Let $T>0$ be the time given by
	Theorem~\ref{thm:cont-reg}, and let the initial approximation be
	$X_N(0):=\Lhyp_N X(0)$.  Then there exists $N_0>0$ such that, for
	$N\geq N_0$, the following error estimate holds:
	\begin{align}\label{eq:final-conv}
			\norm{X_N-X}_{L_t^\infty L_x^2([0,T]\times\Sph)}
			+
			\norm{\nabla_{\Sph} (X_N- X) \cdot n\circ X}_{L_t^2L_x^2([0,T]\times\Sph)}
			\leq C e^{-c N} .
	\end{align}
	The constants $N_0$, $C$, and $c$ depend only on $X$ and $T$.
\end{theorem}
\begin{remark}\upshape
	The order condition $R\geq 3N$ suppresses the aliasing error generated by the
	nonlinearity; see \cite[Section 2.1.2]{CHQZ1988book}.
\end{remark}
At the beginning of the proof, we define the error $e_N(t)=X_N(t)-X_{N,*}(t)$ and impose the following induction hypothesis on $e_N$.
Suppose that $[0,T']$ with $T'\leq T$ is the maximal time interval such that
\begin{align}\label{eq:ind-hypo}
	\norm{e_N}_{L_t^\infty W_x^{1,\infty}([0,T']\times \Sph)}
	\leq 
	N^{-4} .
\end{align}
To conclude the convergence proof, it shall be shown that $T'=T$ alongside the main error estimate \eqref{eq:final-conv} in Section \ref{sec:conv-proof}.


\section{Consistency analysis}\label{sec:cons}
The consistency error $\mathscr D(\cdot)$ is defined by subtracting the continuous weak form \eqref{eq:dziuk-weak} from the scheme \eqref{eq:pdf-semi-scheme}:
\begin{align}\label{eq:cons-eq}
	&\mathscr M_{\Gamma_{N,*}(t),\Q_R}\big(\partial_tX_{N,*}(t),\phi_N\big)
	+ \mathscr M_{\Gamma(t)}\big(H(t)n(t),\phi_N\big)
	\notag\\
	&\quad+
	\mathscr A_{\Gamma_{N,*}(t),\Q_R}\big(X_{N,*}(t),\phi_N\big)
	-\mathscr A_{\Gamma(t)}\big({\rm id},\phi_N\big)
	=: \mathscr D(\phi_N) .
\end{align}
The approximation error estimates in Lemma \ref{lemma:anal-approx} imply the following
exponential accuracy of the geometric perturbation errors between $\Gamma_{N,*}(t)$ and $\Gamma(t)$ whenever the underlying flow is real analytic:

\begin{align}
	&\norm{X_{N,*}(t)-X(t)}_{W^{s,p}(\Sph)}
	+
	\norm{n_{N,*}(t)\circ X_{N,*}(t)-n(t)\circ X(t)}_{W^{s,p}(\Sph)}
	\notag\\
	&
	+
	\norm{J_{N,*}(t)-J(t)}_{W^{s,p}(\Sph)}
	+\norm{A_{N,*}(t)-A(t)}_{W^{s,p}(\Sph)}
	\leq C_{s} e^{-c_s N}. \notag
\end{align}
Consequently, we have the following norm equivalence result whose proof is standard and therefore omitted; cf. \cite[Lemma 7.1]{KLL19}.
\begin{lemma}\label{lemma:norm-equiv}
	For any $f$ defined on $\Gamma_{N,*}$, $s\geq 0$ and $1\leq p\leq \infty$, it holds that
	\begin{align}
		C_s^{-1}
		\| f\circ X_{N,*} \|_{W^{s,p}(\Sph)}
		\leq
		\| f \|_{W^{s,p}(\Gamma_{N,*})}
		\leq
		C_s
		\| f\circ X_{N,*} \|_{W^{s,p}(\Sph)} , \notag
	\end{align}
	for some constant $C_s$ and sufficiently large $N$.
\end{lemma}

\subsection{Mass consistency error}
The consistency error for the mass bilinear form reads:
\begin{align}
  &\mathscr M_{\Gamma_{N,*}(t),\Q_R}\big(\partial_tX_{N,*}(t),\phi_N\big)
    + \mathscr M_{\Gamma(t)}\big(H(t)n(t),\phi_N\big) \notag\\
  &=\Q_R\big(\Lhyp_N\partial_tX(t)\cdot\phi_N
      J_{N,*}(t)\big)
    -\int_{\Sph}\partial_tX(t)\cdot\phi_N J(t) \notag\\
  &= \Q_R \big(\Lhyp_N\partial_tX(t)\cdot\phi_N
      (J_{N,*}(t)-J(t))\big) \notag\\
  &\quad
    -\bigg(\int_{\Sph}-\Q_R \bigg)
      \big(\Lhyp_N\partial_tX(t)\cdot\phi_N J(t)\big) \notag\\
  &\quad
    -\int_{\Sph}(1-\Lhyp_N)\partial_tX(t)\cdot\phi_N J(t) \notag\\
  &=:\mathscr D_{11}(\phi_N)+\mathscr D_{12}(\phi_N)
    +\mathscr D_{13}(\phi_N). \notag
\end{align}
For the first term, we have
\begin{align}
  |\mathscr D_{11}(\phi_N)|
  &= \left|\Q_R\big(\Lhyp_N\partial_tX(t)\cdot\phi_N
      (J_{N,*}(t)-J(t))\big)\right| \notag\\
  &\leq C
      \norm{J_{N,*}(t)-J(t)}_{L^\infty(\Sph)}
      \norm{\Lhyp_N\partial_tX(t)}_{L^\infty(\Sph)}
      \norm{\phi_N}_{L^\infty(\Sph)}
      \notag\\
  &\leq C
      \norm{J_{N,*}(t)-J(t)}_{L^\infty(\Sph)}
      \big(\norm{\partial_tX(t)}_{L^\infty(\Sph)}
        +\norm{(1-\Lhyp_N)\partial_tX(t)}_{L^\infty(\Sph)}\big)
      \norm{\phi_N}_{L^\infty(\Sph)} \notag\\
  &\leq C e^{-cN}\norm{\phi_N}_{L^\infty(\Sph)}. \notag
\end{align}
The second term can be handled as follows:
\begin{align}
  |\mathscr D_{12}(\phi_N)|
  &= \left| -\left(\int_{\Sph}-\Q_R \right)
      \big(\Lhyp_N\partial_tX(t)\cdot\phi_N J(t)\big) \right| \notag\\
  &\leq C
      \inf_{\substack{p_{R-2N}\\ \in \V_{R-2N}}}
      \norm{J(t)-p_{R-2N}}_{L^\infty(\Sph)}
      \norm{\Lhyp_N\partial_tX(t)}_{L^\infty(\Sph)}
      \norm{\phi_N}_{L^\infty(\Sph)} \notag\\
  &\leq C e^{-cN}\norm{\phi_N}_{L^\infty(\Sph)}, \notag
\end{align}
and, for $\mathscr D_{13}$, we estimate
\begin{align}
  |\mathscr D_{13}(\phi_N)|
  &= \left| -\int_{\Sph}(1-\Lhyp_N)\partial_tX(t)\cdot\phi_N J(t) \right|
  \notag\\
  &\leq C
      \inf_{p_N\in\V_N}
      \norm{\partial_tX(t)-p_N}_{L^\infty(\Sph)}
      \norm{J(t)}_{L^\infty(\Sph)}
      \norm{\phi_N}_{L^2(\Sph)} \notag\\
  &\leq C e^{-cN}\norm{\phi_N}_{L^2(\Sph)}. \notag
\end{align}
Since $\partial_t X = (-Hn)\circ X$ is analytic, Lemma \ref{lemma:anal-approx} gives exponential convergence for $\mathscr D_{13}$.

\subsection{Stiffness consistency error}
The consistency error for the stiffness bilinear form can be decomposed as follows:
\begin{align}
  &\mathscr A_{\Gamma_{N,*}(t),\Q_R}\big(X_{N,*}(t),\phi_N\big)
    -\mathscr A_{\Gamma(t)}\big({\rm id},\phi_N\big) \notag\\
  &=\Q_R\big(
      A_{N,*}(t)\gradS(\Lhyp_NX(t))\cdot\gradS\phi_N\big)
    -\int_{\Sph}A(t)\gradS X(t)\cdot\gradS\phi_N \notag\\
  &=\Q_R\big((A_{N,*}(t)-A(t))
      \gradS(\Lhyp_NX(t))\cdot\gradS\phi_N\big) \notag\\
  &\quad
    -\bigg(\int_{\Sph}-\Q_R \bigg)
      \big(A(t)\gradS(\Lhyp_NX(t))\cdot\gradS\phi_N\big) \notag\\
  &\quad
    -\int_{\Sph}A(t)\gradS\big((1-\Lhyp_N)X(t)\big)\cdot\gradS\phi_N
      \notag\\
  &=:\mathscr D_{21}(\phi_N)+\mathscr D_{22}(\phi_N)
    +\mathscr D_{23}(\phi_N). \notag
\end{align}
Here,
\begin{align}
  |\mathscr D_{21}(\phi_N)|
  &=\left|\Q_R\big((A_{N,*}(t)-A(t))
      \gradS(\Lhyp_NX(t))\cdot\gradS\phi_N\big)\right| \notag\\
  &\leq C
      \norm{A_{N,*}(t)-A(t)}_{L^\infty(\Sph)}
      \norm{\gradS(\Lhyp_NX(t))}_{L^\infty(\Sph)}
      \norm{\gradS\phi_N}_{L^\infty(\Sph)} \notag\\
  &\leq C e^{-cN}\norm{\gradS\phi_N}_{L^\infty(\Sph)}. \notag
\end{align}
The quadrature consistency error satisfies
\begin{align}
  |\mathscr D_{22}(\phi_N)|
  &= \left| -\left(\int_{\Sph}-\Q_R \right)
      \big(A(t)\gradS(\Lhyp_NX(t))\cdot\gradS\phi_N\big) \right|
      \notag\\
  &\leq C
      \inf_{\substack{p_{R-2N-2}\\ \in\V_{R-2N-2}}}
      \norm{A(t)-p_{R-2N-2}}_{L^\infty(\Sph)}
      \norm{\gradS\Lhyp_NX(t)}_{L^\infty(\Sph)}
      \norm{\gradS\phi_N}_{L^\infty(\Sph)} \notag\\
  &\leq C e^{-cN}\norm{\gradS\phi_N}_{L^\infty(\Sph)}. \notag
\end{align}
Finally, we have
\begin{align}
  |\mathscr D_{23}(\phi_N)|
  &= \left| -\int_{\Sph}
      A(t)\gradS\big((1-\Lhyp_N)X(t)\big)\cdot\gradS\phi_N \right|
      \notag\\
  &\leq C
      \norm{\nabla_{\Sph}(1-\Lhyp_N)X(t)}_{L^2(\Sph)}
      \norm{A(t)}_{L^\infty(\Sph)}
      \norm{\gradS\phi_N}_{L^2(\Sph)} \notag\\
  &\leq C e^{-cN}\norm{\gradS\phi_N}_{L^2(\Sph)} . \notag
\end{align}
In summary, 
collecting all the estimates above, we obtain
\begin{align}\label{eq:D-est}
	| \mathscr D(\phi_N) |
	\leq C e^{-cN}\norm{\phi_N}_{L^2(\Sph)},
\end{align}
where we have used the inverse inequality and absorbed the factor $N^2$ into the exponential rate by adjusting the generic constants $C$ and $c$ suitably, i.e.,
\begin{align}
	C e^{-cN}\norm{\phi_N}_{W^{1,\infty}(\Sph)}
	\leq
	C N^{2} e^{-cN} \norm{\phi_N}_{L^{2}(\Sph)}
	\leq
	C e^{-c N} \norm{\phi_N}_{L^{2}(\Sph)}. \notag
\end{align}

\section{Stability analysis}\label{sec:stab}

We first define the intermediate surface $\Gamma_{N,\theta}(t)$ and its parametrization $X_{N,\theta}(t)\in[\V_N]^3$:
\begin{equation*}
	X_{N,\theta}(t)=(1-\theta)X_{N,*}(t)+\theta X_N(t),
	\qquad
	\Gamma_{N,\theta}(t)=X_{N,\theta}(t)(\Sph).
\end{equation*}
For $\Gamma_{N,\theta}(t)$, we write the corresponding pullback coefficients of mass and stiffness bilinear forms as
$J_{N,\theta}(t)$ and $A_{N,\theta}(t)$. 
As in \eqref{eq:JA-bar}, $J_{N,\theta}(t)$ and $A_{N,\theta}(t)$ take the form
\begin{align}
	J_{N,\theta}(t,x) = \bar J(t, \nabla_{\Sph} X_{N,\theta}(t, x)), \qquad
	A_{N,\theta}(t,x) = \bar A(t, \nabla_{\Sph} X_{N,\theta}(t, x)). \notag
\end{align}
Note that $\bar J$ and $\bar A$ are smooth around the exact solution.
If we denote the second argument of $\bar J(\cdot,\cdot)$ and $\bar A(\cdot,\cdot)$ as variable $G$, then we have the following chain rule:
	\begin{align}
		\frac{\dd}{\dd\theta}\bar J(t,\gradS X_{N,\theta})
		=\nabla_G\bar J(t,\gradS X_{N,\theta})
		[\frac{\dd}{\dd\theta}\gradS X_{N,\theta}]
		=\nabla_G\bar J(t,\gradS X_{N,\theta})\,[\gradS e_N]
		, \notag\\
		\frac{\dd}{\dd\theta}\bar A(t,\gradS X_{N,\theta})
		=\nabla_G\bar A(t,\gradS X_{N,\theta})
		[\frac{\dd}{\dd\theta}\gradS X_{N,\theta}]
		=\nabla_G\bar A(t,\gradS X_{N,\theta})\,[\gradS e_N] , \notag
	\end{align}
	where, by our notation, $[\cdot]$ contracts with $\nabla_G$ into a scalar.
		Moreover, by the Lipschitz continuity of $\bar J$ and $\bar A$,
	\begin{align}
		\norm{\bar J(t,\gradS X_{N,\theta})-\bar J(t,\gradS X_{N,*})}_{L^\infty(\Sph)}
		\leq C \norm{\theta \nabla_{\Sph}e_N(t)}_{L^\infty(\Sph)}
		, \\
		\norm{\bar A(t,\gradS X_{N,\theta})-\bar A(t,\gradS X_{N,*})}_{L^\infty(\Sph)}
		\leq C \norm{\theta \nabla_{\Sph}e_N(t)}_{L^\infty(\Sph)}. \label{eq:Abar-Lip}
		\end{align}

\subsection{Mass error equation}

The error of mass bilinear form can be decomposed as
\begin{align}
  &\mathscr M_{\Gamma_{N},\Q_R}(\partial_tX_N,\phi_N)
    -
    \mathscr M_{\Gamma_{N,*},\Q_R}(\partial_tX_{N,*},\phi_N) \notag\\
  &=\Q_R\big(\partial_tX_N\cdot\phi_N J_N(t)\big)
    -\Q_R\big(\partial_tX_{N,*}\cdot\phi_N
      J_{N,*}(t)\big) \notag\\
  &=\int_{\Sph}\partial_te_N\cdot\phi_N J_{N,*}(t)
    -\bigg(\int_{\Sph}-\Q_R \bigg)\big(\partial_te_N\cdot\phi_N
      J_{N,*}(t)\big) \notag\\
  &\quad
    +\Q_R\big(\partial_tX_N\cdot\phi_N
      (J_N(t)-J_{N,*}(t))\big) \notag\\
  &=:\int_{\Gamma_{N,*}(t)}\partial_te_N\cdot\phi_N
    +\mathscr J_1(\phi_N)+\mathscr J_2(\phi_N). \notag
\end{align}
The quadrature error $\mathscr J_1$ satisfies
\begin{align}
  |\mathscr J_1(\phi_N)|
  &=\left|-\left(\int_{\Sph}-\Q_R \right)\big(\partial_te_N\cdot\phi_N
      J_{N,*}(t)\big) \right| \notag\\
  &= \bigg| -\left(\int_{\Sph}-\Q_R \right)(\partial_te_N\cdot\phi_N J(t)) \notag\\
  &\quad
    -\left(\int_{\Sph}-\Q_R \right)\big(\partial_te_N\cdot\phi_N
      (J_{N,*}(t)-J(t))\big) \bigg| \notag\\
  &\leq C
      \bigg(\inf_{\substack{p_{R-2N}\\ \in\V_{R-2N}}}
      \norm{J(t)-p_{R-2N}}_{L^\infty(\Sph)}
      +
      \norm{J_{N,*}(t)-J(t)}_{L^\infty(\Sph)}\bigg)
       \notag\\
  &\quad
  \times
      \norm{\partial_te_N}_{L^\infty(\Sph)}
      \norm{\phi_N}_{L^\infty(\Sph)}
       \notag\\
  &\leq C e^{-cN}
      \norm{\partial_te_N}_{L^\infty(\Sph)} \norm{\phi_N}_{L^\infty(\Sph)} . \notag
\end{align}
$\mathscr J_2$ can be further decomposed as:
\begin{align}
  \mathscr J_2(\phi_N)
  &=  \Q_R\big(\partial_tX_N\cdot\phi_N
      (J_N(t)-J_{N,*}(t))\big)  \notag\\
  &=\int_{\Sph}\partial_tX(t)\cdot\phi_N
      (J_N(t)-J_{N,*}(t)) \notag\\
  &\quad
    -\left(\int_{\Sph}-\Q_R \right)\big(\partial_tX(t)\cdot\phi_N
      (J_N(t)-J_{N,*}(t))\big) \notag\\
  &\quad+
    \Q_R\big(\partial_t(X_N-X)(t)\cdot\phi_N
      (J_N(t)-J_{N,*}(t))\big)  \notag\\
	&=: \mathscr J_{21}(\phi_N) + \mathscr J_{22}(\phi_N) + \mathscr J_{23}(\phi_N)
      . \notag
\end{align}
The domain discrepancy error \(\mathscr J_{21}\) is estimated by
\begin{align}\label{eq:pdf-J2-geometric-difference}
	|\mathscr J_{21}(\phi_N)|
	&=
  \left|\int_{\Sph}\partial_tX(t)\cdot\phi_N
      (J_N(t)-J_{N,*}(t)) \right| \notag\\
  &=\left| \int_0^1\frac{\dd}{\dd\theta}
      \int_{\Sph}\partial_tX(t)\cdot\phi_N J_{N,\theta}(t)
      \,\dd\theta \right| \notag\\
  &= \left| \int_0^1\int_{\Sph}\partial_tX(t)\cdot\phi_N
      \nabla_G\bar J(t,\gradS X_{N,\theta})[\gradS e_N]
      \,\dd\theta \right| \notag\\
  &= \bigg| \int_0^1\int_{\Sph}\partial_tX(t)\cdot\phi_N
      \nabla_G\bar J(t,\gradS X)[\gradS e_N]
      \,\dd\theta \notag\\
  &\quad+
    \int_0^1\int_{\Sph}\partial_tX(t)\cdot\phi_N
      \big(\nabla_G\bar  J(t,\gradS X_{N,\theta})
      - \nabla_G\bar  J(t,\gradS X)\big)[\gradS e_N]
      \,\dd\theta \bigg| \notag\\
  &\leq C
      (\norm{\phi_N}_{L^2(\Sph)} + \norm{\gradS\phi_N\cdot n\circ X}_{L^2(\Sph)})
      \norm{e_N}_{L^2(\Sph)} \notag\\
  &\quad+
      \big(C e^{-cN}+\norm{\gradS e_N}_{L^\infty(\Sph)}\big)
      \norm{\gradS e_N}_{L^\infty(\Sph)}
      \norm{\phi_N}_{L^2(\Sph)}
       ,
\end{align}
where the last inequality uses integration by parts on $\Sph$ (cf. \cite[Lemma 5.1]{BL2024}) to move the gradient of $\bar J$ to the front of $\phi_N$, and the relation $\partial_tX = (-Hn)\circ X$, yielding the term
$\nabla_{\Sph}\phi_N\cdot n\circ X$ up to an $L^2$ remainder.
This directional $H^1$ term is crucial and can be controlled by the normal parabolicity in Section \ref{sec:norm-parab}.

For $\mathscr J_{22}$, we have
\begin{align}\label{eq:pdf-J22}
	|\mathscr J_{22}(\phi_N)|
	&= \left| -\left(\int_{\Sph}-\Q_R \right)
	\big(\partial_t X(t)\cdot\phi_N (J_N(t)-J_{N,*}(t)) \big) \right|
	\notag\\
	&= \left| -\int_0^1\frac{\dd}{\dd\theta}
	\left(\int_{\Sph}-\Q_R \right)
	\big(\partial_t X\cdot\phi_N J_{N,\theta}(t)\big)
	\,\dd\theta \right| \notag\\
	&= \bigg| -\int_0^1\left(\int_{\Sph}-\Q_R \right)
	\big(	\partial_t X\cdot\phi_N \nabla_G\bar J(t,\gradS X_{N,\theta})[\gradS e_N]\,
	\big)\,\dd\theta \bigg| \notag\\
	&\leq C
	\bigg(
	\inf_{\substack{p_{R-2N-1}\\ \in\V_{R-2N-1}}}
	\norm{\partial_t X\nabla_G\bar J(t,\nabla_{\Sph}X)-p_{R-2N-1}}_{L^\infty(\Sph)}
	\notag\\
	&\quad
	+
	\norm{\nabla_G\bar J(t,\gradS X_{N,\theta}) 
		- \nabla_G\bar J(t,\gradS X) }
	_{L^\infty(\Sph)}
	\norm{\partial_t X}_{L^\infty(\Sph)}
	\bigg) \notag\\
	&\quad\times
	\norm{ \gradS e_N}_{L^\infty(\Sph)}
	\norm{ \phi_N}_{L^\infty(\Sph)} \notag\\
	&\leq C
	\big( e^{-cN}+\norm{\gradS e_N}_{L^\infty(\Sph)}\big)
	\norm{\gradS e_N}_{L^\infty(\Sph)} \norm{ \phi_N}_{L^\infty(\Sph)} ,
\end{align}
and H\"older's inequality gives the following estimate for $\mathscr J_{23}$:
\begin{align}
	|\mathscr J_{23}(\phi_N)|
	&\leq C
	 (e^{-cN} + \norm{\partial_t e_N}_{L^\infty(\Sph)})
	\norm{\gradS e_N}_{L^\infty(\Sph)}
	\norm{\phi_N}_{L^\infty(\Sph)} . \notag
\end{align}

In summary, if we define
\begin{equation*}
	\mathscr J(\phi_N)
	:=\mathscr J_{1}(\phi_N)+\mathscr J_{2}(\phi_N),
\end{equation*}
then it satisfies
\begin{align}
	| \mathscr J(\phi_N) | 
	&\leq
	C e^{-cN}
	\norm{\partial_te_N}_{L^\infty(\Sph)} \norm{\phi_N}_{L^\infty(\Sph)}
	\notag\\
	&\quad+
	C
	\big(e^{-cN} + \norm{\gradS e_N}_{L^\infty(\Sph)} + \norm{\partial_t e_N}_{L^\infty(\Sph)}\big)
	\norm{\gradS e_N}_{L^\infty(\Sph)}
	\norm{\phi_N}_{L^\infty(\Sph)}
	\notag\\
	&\quad+
	(\norm{\phi_N}_{L^2(\Sph)} + \norm{\gradS\phi_N\cdot n\circ X}_{L^2(\Sph)})
	\norm{e_N}_{L^2(\Sph)}
	\notag\\
	&\quad+
	\big(C e^{-cN}+\norm{\gradS e_N}_{L^\infty(\Sph)}\big)
	\norm{\gradS e_N}_{L^\infty(\Sph)} \norm{\phi_N}_{L^2(\Sph)} .
	\label{eq:J-est}
\end{align}

\subsection{Stiffness error equation}

For the stiffness part, we have
\begin{align}
  &\mathscr A_{\Gamma_N,\Q_R}(X_N,\phi_N)
    -\mathscr A_{\Gamma_{N,*},\Q_R}(X_{N,*},\phi_N) \notag\\
  &=\Q_R\big(A_N(t)\gradS X_N\cdot\gradS\phi_N\big)
    -\Q_R\big(A_{N,*}(t)\gradS X_{N,*}\cdot\gradS\phi_N\big)
       \notag\\
  &=\int_{\Sph}A_N(t)\gradS X_N\cdot\gradS\phi_N
    -\int_{\Sph}A_{N,*}(t)\gradS X_{N,*}\cdot\gradS\phi_N \notag\\
  &\quad
    -\bigg(\int_{\Sph}-\Q_R \bigg)
      \big(A_N(t)\gradS X_N\cdot\gradS\phi_N\big) \notag\\
  &\quad
    +\bigg(\int_{\Sph}-\Q_R \bigg)
      \big(A_{N,*}(t)\gradS X_{N,*}\cdot\gradS\phi_N\big) \notag\\
  &=\int_{\Sph}A_N(t)\gradS X_N\cdot\gradS\phi_N
    -\int_{\Sph}A_{N,*}(t)\gradS X_{N,*}\cdot\gradS\phi_N \notag\\
  &\quad
    -\bigg(\int_{\Sph}-\Q_R \bigg)
      \big(A_N(t)\gradS e_N\cdot\gradS\phi_N\big) \notag\\
  &\quad
    -\bigg(\int_{\Sph}-\Q_R \bigg)
      \big((A_N(t)-A_{N,*}(t))\gradS X_{N,*}\cdot\gradS\phi_N\big)
      \notag\\
  &=:\int_{\Gamma_N}\nabla_{\Gamma_N}X_N\cdot\nabla_{\Gamma_N}\phi_N
    -\int_{\Gamma_{N,*}}\nabla_{\Gamma_{N,*}}X_{N,*}\cdot
      \nabla_{\Gamma_{N,*}}\phi_N \notag\\
  &\quad+\mathscr K_1(\phi_N)+\mathscr K_2(\phi_N). \notag
\end{align}
The term \(\mathscr K_1\) satisfies
\begin{align}
  |\mathscr K_1(\phi_N)|
  &= \bigg| -\left(\int_{\Sph}-\Q_R \right)
	\big(A_N(t)\gradS e_N\cdot\gradS\phi_N\big) \bigg| \notag\\
  &\leq C
      \bigg(
      \norm{A_N(t)-A(t)}_{L^\infty(\Sph)}
      +
      \inf_{\substack{p_{R-2N-2}\\ \in\V_{R-2N-2}}}
      \norm{A(t)-p_{R-2N-2}}_{L^\infty(\Sph)}
      \bigg)
      \notag\\
  &\quad\times
      \|\gradS e_N\|_{L^\infty(\Sph)} 
      \|\gradS\phi_N\|_{L^\infty(\Sph)}
      \notag\\
  &\leq C
      \big(e^{-cN}
      +\norm{\gradS e_N}_{L^\infty(\Sph)}\big)
     \|\gradS e_N\|_{L^\infty(\Sph)} 
     \|\gradS\phi_N\|_{L^\infty(\Sph)} , \notag
\end{align}
and, for $\mathscr K_2$, we have
\begin{align}\label{eq:pdf-K2}
  |\mathscr K_2(\phi_N)|
  &= \left| -\left(\int_{\Sph}-\Q_R \right)
      \big((A_N(t)-A_{N,*}(t))\gradS X_{N,*}\cdot\gradS\phi_N\big) \right|
      \notag\\
  &= \left| -\int_0^1\frac{\dd}{\dd\theta}
      \left(\int_{\Sph}-\Q_R \right)
      \big(A_{N,\theta}(t)\gradS X_{N,*}\cdot\gradS\phi_N\big)
      \,\dd\theta \right| \notag\\
  &= \bigg| -\int_0^1\left(\int_{\Sph}-\Q_R \right)
      \big(\nabla_G\bar A(t,\gradS X_{N,\theta})[\gradS e_N]\,
      \gradS X_{N,*}\cdot\gradS\phi_N\big)\,\dd\theta \bigg| \notag\\
  &\leq C
      \bigg(
      \inf_{\substack{p_{R-2N-2}\\ \in\V_{R-2N-2}}}
      \norm{\nabla_G\bar A(t,\nabla_{\Sph}X)
      \nabla_{\Sph}X-p_{R-2N-2}}_{L^\infty(\Sph)}
            \notag\\
      &\quad
      +
      \norm{\nabla_G\bar A(t,\gradS X_{N,\theta}) \nabla_{\Sph}X_{N,*}
      - \nabla_G\bar A(t,\gradS X) \nabla_{\Sph}X}
        _{L^\infty(\Sph)}
      \bigg) \notag\\
  &\quad\times
  \norm{\gradS e_N}_{L^\infty(\Sph)}
      \norm{ \gradS\phi_N}_{L^\infty(\Sph)} \notag\\
  &\leq C
      \big( e^{-cN}+\norm{\gradS e_N}_{L^\infty(\Sph)}\big)
      \norm{\gradS e_N}_{L^\infty(\Sph)} \norm{ \gradS\phi_N}_{L^\infty(\Sph)} ,
\end{align}
where, in the last line, we have used the Lipschitz continuity estimate \eqref{eq:Abar-Lip}.

Combining the two estimates above and defining
\begin{equation*}
	\mathscr K(\phi_N)
	:=\mathscr K_{1}(\phi_N)+\mathscr K_{2}(\phi_N),
\end{equation*}
we finally arrive at
	\begin{align}\label{eq:K-est}
		| \mathscr K(\phi_N) | 
	&\leq 
		C\big( e^{-cN}+\norm{\gradS e_N}_{L^\infty(\Sph)}\big)
		\norm{\gradS e_N}_{L^\infty(\Sph)}
		\norm{\gradS \phi_N}_{L^\infty(\Sph)}.
	\end{align}

\subsection{Normal parabolicity}\label{sec:norm-parab}

For $\R^3$-valued functions $u,v$ on a smooth surface $\Sigma$, we define
\begin{align}
  \mathscr A_{\Sigma}(u,v)
  &:=\int_{\Sigma}\nabla_{\Sigma}u\cdot\nabla_{\Sigma}v,\notag\\
  \mathscr A^{\rm Nor}_{\Sigma}(u,v)
  &:=\int_{\Sigma}
      [(\nabla_{\Sigma}u)n]\cdot[(\nabla_{\Sigma}v)n],\notag\\
  \mathscr A^{\rm Tan}_{\Sigma}(u,v)
  &:=\int_{\Sigma}
      \tr\!\big((\nabla_{\Sigma}u)(I-nn^\top)
      (\nabla_{\Sigma}v)^\top\big),\notag\\
  \mathscr B_{\Sigma}(u,v)
  &:=\int_{\Sigma}
      (\nabla_{\Sigma}\cdot u)(\nabla_{\Sigma}\cdot v)
      -\tr(\nabla_{\Sigma}u\nabla_{\Sigma}v). \notag
\end{align}
Thus $\mathscr A_\Sigma=\mathscr A_\Sigma^{\rm Nor}
+\mathscr A_\Sigma^{\rm Tan}$.
We define the surface partial derivative $\ud_i u := (\nabla_\Sigma u)_i $, $i=1,2,3$.
Using the Einstein summation convention, with
\begin{equation*}
  (D_\Sigma u)_{rl}:=-\ud_lu_r-\ud_ru_l+\delta_{rl}\ud_m u_m ,
\end{equation*}
and $P=I-nn^\top=\nabla_\Sigma{\rm id}$, we have the following identity (see \cite[Eq. (5.8)]{BL2024})
\begin{align}\label{eq:pdf-DGamma-identity}
  \int_\Sigma\nabla_\Sigma{\rm id}\cdot(D_\Sigma u)\nabla_\Sigma v
  &=-\mathscr A_\Sigma^{\rm Tan}(u,v)+\mathscr B_\Sigma(u,v).
\end{align}
The tensor $D_\Sigma u$ plays an important role in the lemma below; see also
\cite[Lemma 7.1]{KLL19}.
\begin{lemma}\label{lemma:A_iden}
	For any $f_N,g_N\in\V_N$, it holds that
	\begin{align}
		&\int_{\Gamma_N} \nabla_{\Gamma_N} f_N \cdot
		\nabla_{\Gamma_N} g_N
		-\int_{\Gamma_{N,*}} \nabla_{\Gamma_{N,*}} f_N \cdot
		\nabla_{\Gamma_{N,*}} g_N \notag\\
		&\qquad =
		\int_0^1 \int_{\Gamma_{N,\theta}}
		\nabla_{\Gamma_{N,\theta}} f_N \cdot
		(D_{\Gamma_{N,\theta}} e_N)
		\nabla_{\Gamma_{N,\theta}} g_N\,\d\theta. \notag
	\end{align}
\end{lemma}
Lemma \ref{lemma:A_iden} and the identity \eqref{eq:pdf-DGamma-identity} immediately lead to the following computations
\begin{align}
	&\int_{\Gamma_N} \nabla_{\Gamma_N} X_N \cdot
	\nabla_{\Gamma_N}\phi_N
	-\int_{\Gamma_{N,*}} \nabla_{\Gamma_{N,*}} X_{N,*} \cdot
	\nabla_{\Gamma_{N,*}}\phi_N \notag\\
	&=\int_{\Gamma_{N,\theta}}
	\nabla_{\Gamma_{N,\theta}} X_{N,\theta} \cdot
	\nabla_{\Gamma_{N,\theta}}\phi_N
	\bigg|_{\theta=0}^{\theta=1} \notag\\
	&=\int_0^1\frac{\d}{\d\theta}
	\int_{\Gamma_{N,\theta}}
	\nabla_{\Gamma_{N,\theta}} X_{N,\theta} \cdot
	\nabla_{\Gamma_{N,\theta}}\phi_N\,\d\theta \notag\\
	&=\int_0^1 \int_{\Gamma_{N,\theta}}
	\nabla_{\Gamma_{N,\theta}} e_N \cdot
	\nabla_{\Gamma_{N,\theta}}\phi_N\,\d\theta \notag\\
	&\quad+
	\int_0^1 \int_{\Gamma_{N,\theta}}
	\nabla_{\Gamma_{N,\theta}} X_{N,\theta} \cdot
	D_{\Gamma_{N,\theta}} e_N
	\nabla_{\Gamma_{N,\theta}}\phi_N\,\d\theta  \notag\\
	&= \int_0^1
	\big[
	\mathscr A_{\Gamma_{N,\theta}}(e_N,\phi_N)
	- \mathscr A_{\Gamma_{N,\theta}}^{\rm Tan}(e_N,\phi_N)
	+ \mathscr B_{\Gamma_{N,\theta}}(e_N,\phi_N)
	\big]\,\d\theta \notag\\
	&= \mathscr A_{\Gamma_{N,*}}^{\rm Nor} (e_N,\phi_N)
	+ 
	\mathscr B_{\Gamma_{N,*}}(e_N,\phi_N) + \mathscr L(\phi_N),
	\label{eq:norm_para}
\end{align}
where, for simplicity, we have introduced the notation
\begin{align}
	\mathscr L(\phi_N)
	&:= \int_0^1
	\big[
	\mathscr A_{\Gamma_{N,\theta}}^{\rm Nor}(e_N,\phi_N)
	- \mathscr A_{\Gamma_{N,*}}^{\rm Nor}(e_N,\phi_N)
	\big]\,\d\theta  \notag\\
	&\quad+ \int_0^1
	\big[
	\mathscr B_{\Gamma_{N,\theta}}(e_N,\phi_N)
	- \mathscr B_{\Gamma_{N,*}}(e_N,\phi_N)
	\big]\,\d\theta. \notag
\end{align}
The remainder $\mathscr L$ is a higher-order error term and can be estimated in a standard way (cf. \cite[Eq. (5.22)]{BL2024}).  Applying the
fundamental theorem of calculus once more, together with norm equivalence, yields
\begin{align}
	|\mathscr L(\phi_N)|
	\leq C
	(e^{-cN}+\norm{\nabla_{\Sph}e_N}_{L^\infty(\Sph)})
	\norm{\nabla_{\Sph}e_N}_{L^2(\Sph)}
	\norm{\nabla_{\Sph}\phi_N}_{L^2(\Sph)}. \notag
\end{align}
Moreover, an integration-by-parts argument (cf. \cite[Eq. (2.1)]{BL22A} and
\cite[Eq. (5.9)]{BL2024}) gives
\begin{align}\label{eq:B-iden}
	\mathscr B_{\Gamma_{N,*}}(e_N, \phi_N)
	&= \int_{\Gamma_{N,*}}
	e_{N,j}\,\ud_i\phi_{N,i}\,H_{N,*}(n_{N,*})_j
	-
	\int_{\Gamma_{N,*}}
	e_{N,j}\,\ud_j\phi_{N,i}\,H_{N,*}(n_{N,*})_i \notag\\
	&\quad
	+\int_{\Gamma_{N,*}}
	e_{N,j}\,\ud_k\phi_{N,i}\,(n_{N,*})_i (h_{N,*})_{jk} 
	-\int_{\Gamma_{N,*}}
	e_{N,j}\,\ud_k\phi_{N,i}\,(n_{N,*})_j (h_{N,*})_{ik},
\end{align}
where $n_{N,*}$, $h_{N,*}$ and $H_{N,*}:= \sum_i(h_{N,*})_{ii}$ are the outward normal vector, the second fundamental form, and the mean curvature
on $\Gamma_{N,*}$.
For $\mathscr A_{\Gamma_{N,*}}^{\rm Nor}$ and $\mathscr B_{\Gamma_{N,*}}$,
H\"older's inequality immediately gives
\begin{align}
		 |\mathscr A_{\Gamma_{N,*}}^{\rm Nor}(e_N,\phi_N)|
		\leq C
		\norm{\nabla_{\Sph}e_N}_{L^2(\Sph)}
		\norm{\nabla_{\Sph}\phi_N}_{L^2(\Sph)} , \notag
\end{align}
and, with integration by parts,
\begin{align}
	 |\mathscr B_{\Gamma_{N,*}}(e_N,\phi_N)|
	\leq C
	\min\left\{
	\norm{e_N}_{L^2(\Sph)} \norm{\phi_N}_{H^1(\Sph)},
	\norm{e_N}_{H^1(\Sph)} \norm{\phi_N}_{L^2(\Sph)}
	\right\}. \notag
\end{align}
When choosing $\phi_N=e_N$, each term on the right-hand side of \eqref{eq:B-iden}
contains either the factor $e_N\cdot n_{N,*}$ or
$\nabla_{\Gamma_{N,*}}e_N\cdot n_{N,*}$.  Therefore, integration by
parts in this case gives a finer bound
\begin{align}\label{eq:B-est1}
	| \mathscr B_{\Gamma_{N,*}}(e_N, e_N) | 
	&\leq C \| e_N \|_{L^2(\Sph)}^2
	+ C \| \nabla_{\Sph} e_N \cdot n_{N,*}\circ X_{N,*} \|_{L^2(\Sph)}
	\| e_N \|_{L^2(\Sph)}
	\notag\\
	&\leq C \| e_N \|_{L^2(\Sph)}^2
	+ C \| \nabla_{\Sph} e_N \cdot n \circ X \|_{L^2(\Sph)}
	\| e_N \|_{L^2(\Sph)}
	\notag\\
	&\quad
	+
	C e^{-cN}
	\|\nabla_{\Sph} e_N \|_{L^2(\Sph)} \| e_N \|_{L^2(\Sph)}
	,
\end{align}
where we have used
\begin{align}\label{eq:H1-norm-convert}
	&\bigg|
	\| \nabla_{\Sph} e_N \cdot n_{N,*}\circ X_{N,*} \|_{L^2(\Sph)}
	-
	\|\nabla_{\Sph} e_N \cdot n \circ X \|_{L^2(\Sph)}
	\bigg|
	\leq C 
	e^{-cN}
	\|\nabla_{\Sph} e_N \|_{L^2(\Sph)}.
\end{align}
Using \eqref{eq:H1-norm-convert} and the norm equivalence (Lemma \ref{lemma:norm-equiv}), we obtain
\begin{align}\label{eq:A-nor-equiv1}
	\| \nabla_{\Sph} e_N \cdot n \circ X \|_{L^2(\Sph)}^2 
	&\leq
	\| \nabla_{\Sph} e_N \cdot n_{N,*}\circ X_{N,*} \|_{L^2(\Sph)}^2
	+
	C
	e^{-cN}
	\|\nabla_{\Sph} e_N \|_{L^2(\Sph)}^2
	\notag\\
	&
	\leq
	C\mathscr A_{\Gamma_{N,*}}^{\rm Nor} (e_N,e_N)
	+
	C
	e^{-cN}
	\|\nabla_{\Sph} e_N \|_{L^2(\Sph)}^2
	.
\end{align}
Thus, the positive definite term
$\mathscr A_{\Gamma_{N,*}}^{\rm Nor} (e_N,e_N)$ can be replaced by
$\| \nabla_{\Sph} e_N \cdot n \circ X \|_{L^2(\Sph)}^2$, up to a harmless
higher-order error.

\section{Convergence proof}\label{sec:conv-proof}

Subtracting \eqref{eq:cons-eq} from \eqref{eq:pdf-semi-scheme}, we obtain the error equation:
\begin{align}\label{eq:error-eq}
	&\int_{\Gamma_{N,*}}\partial_te_N\cdot \phi_N
	+
	\mathscr A_{\Gamma_{N,*}}^{\rm Nor}(e_N,\phi_N) 
	\notag\\
	&\qquad
	=
	-\mathscr B_{\Gamma_{N,*}}(e_N,\phi_N)
	-\mathscr L(\phi_N)-\mathscr K(\phi_N)-\mathscr J(\phi_N)-\mathscr D(\phi_N).
\end{align}
We first test \eqref{eq:error-eq} with $\phi_N=\partial_t e_N\in[\V_N]^3$ to obtain the velocity error estimate
\begin{align}
	&\int_{\Gamma_{N,*}}\partial_te_N\cdot \partial_t e_N
	=
	-\mathscr A_{\Gamma_{N,*}}^{\rm Nor}(e_N,\partial_t e_N) 
	-\mathscr B_{\Gamma_{N,*}}(e_N,\partial_t e_N)
	\notag\\
	&\qquad
	-\mathscr L(\partial_t e_N)-\mathscr K(\partial_t e_N)-\mathscr J(\partial_t e_N)-\mathscr D(\partial_t e_N). \notag
\end{align}
Using H\"older's inequality, norm equivalence, the consistency estimate
\eqref{eq:D-est} from Section \ref{sec:cons}, and the stability estimates
\eqref{eq:J-est}--\eqref{eq:K-est} from Section \ref{sec:stab}
for the linear and bilinear forms on the right-hand side, we arrive at
\begin{align}
	\norm{\partial_t e_N}_{L^2(\Sph)}^2
	&\leq
	C e^{-cN}
	\norm{\partial_te_N}_{L^\infty(\Sph)} \norm{\partial_t e_N}_{L^\infty(\Sph)}
	\notag\\
	&\quad+
	C 
	\big(e^{-c N} + \norm{\gradS e_N}_{L^\infty(\Sph)} + \norm{\partial_t e_N}_{L^\infty(\Sph)}\big)
	\norm{\gradS e_N}_{L^\infty(\Sph)}
	\norm{\partial_t e_N}_{L^\infty(\Sph)}
	\notag\\
	&\quad
	+ 
	C\big(e^{-c N} + \norm{\gradS e_N}_{L^\infty(\Sph)}\big)
	\norm{\gradS e_N}_{L^{\infty}(\Sph)}
	\norm{ \partial_t e_N}_{W^{1,\infty}(\Sph)}
	\notag\\
	&\quad
	+
	 C\norm{ e_N}_{H^1(\Sph)}
	 \norm{ \partial_t e_N}_{H^{1}(\Sph)}. \notag
\end{align}
From the induction hypothesis \eqref{eq:ind-hypo}, the inverse inequality, and
absorption, we obtain
\begin{align}\label{eq:vel-est}
	\norm{\partial_t e_N}_{L^2(\Sph)}
	\leq
	C e^{-c N} 
	+ 
	C N \norm{ e_N}_{H^{1}(\Sph)} ,
\end{align}
for sufficiently large $N$.

Next, taking $\phi_N=e_N\in[\V_N]^3$ in the error equation \eqref{eq:error-eq} gives
\begin{align}
  &\int_{\Gamma_{N,*}}\partial_te_N\cdot e_N
    +\mathscr A_{\Gamma_{N,*}}^{\rm Nor}(e_N,e_N) \notag\\
  &\qquad
  =-\mathscr B_{\Gamma_{N,*}}(e_N,e_N)-\mathscr L(e_N)-\mathscr K(e_N)-\mathscr J(e_N)-\mathscr D(e_N). \notag
\end{align}
Note that we have the identity
\begin{align}\label{eq:ev}
	\int_{\Gamma_{N,*}(t)}\partial_t e_N(t)\cdot e_N(t)
	=
	\frac{1}{2}\frac{\dd}{\dd t}
	\|e_N(t)\|_{L^2(\Gamma_{N,*}(t))}^2
	-
	\frac{1}{2}\int_{\Sph}|e_N(t)|^2\,\partial_tJ_{N,*}(t),
\end{align}
with
\begin{align}\label{eq:ev1}
	\left|
	\frac{1}{2}\int_{\Sph}|e_N(t)|^2\,\partial_tJ_{N,*}(t)
	\right|
	\leq
	C(1 + e^{-cN})\|e_N(t)\|_{L^2(\Sph)}^2
	\leq
	C\|e_N(t)\|_{L^2(\Sph)}^2.
\end{align}
Consequently, using \eqref{eq:A-nor-equiv1}, \eqref{eq:ev}--\eqref{eq:ev1} and estimates for the linear and bilinear forms, we derive
\begin{align}\label{eq:pdf-final-error-eq1}
	&\frac{\d}{\d t}\norm{e_N}_{L^2(\Gamma_{N,*}(t))}^2
	+
	\norm{\nabla_{\Sph} e_N\cdot n\circ X}_{L^2(\Sph)}^2
	 \notag\\
	&\qquad
		\leq 
		C\|e_N\|_{L^2(\Sph)}^2
		+
		C e^{-c N}
	\norm{ e_N}_{L^{2}(\Sph)}
	\notag\\
	&\qquad+
	C e^{-cN}
	\norm{\partial_te_N}_{L^\infty(\Sph)} \norm{e_N}_{L^\infty(\Sph)}
	\notag\\
	&\qquad
	+
	C 
	\big(e^{-c N} + \norm{\gradS e_N}_{L^\infty(\Sph)} + \norm{\partial_t e_N}_{L^\infty(\Sph)}\big)
	\norm{\gradS e_N}_{L^\infty(\Sph)}
	\norm{ e_N}_{L^\infty(\Sph)}
	\notag\\
	&\qquad
	+
	C\norm{e_N}_{L^2(\Sph)} \norm{\nabla_{\Sph} e_N\cdot n\circ X}_{L^2(\Sph)}
	\notag\\
	&\qquad
	+
		C( e^{-c N} + \norm{\nabla_{\Sph} e_N}_{L^{\infty}(\Sph)})
		 \norm{ e_N}_{W^{1,\infty}(\Sph)}^2.
	\end{align}
Here the critical term $\norm{\nabla_{\Sph} e_N\cdot n\circ X}_{L^2(\Sph)}$ comes from \eqref{eq:pdf-J2-geometric-difference} and \eqref{eq:B-est1}.
Using H\"older's inequality,
the velocity estimate \eqref{eq:vel-est}, and the induction hypothesis
\eqref{eq:ind-hypo}, all terms on the right-hand side can either be absorbed or
controlled by the left-hand side.
For instance, the last term on the right-hand side of \eqref{eq:pdf-final-error-eq1} can be bounded as follows:
\begin{align}
	C( e^{-c N} + \norm{\nabla_{\Sph} e_N}_{L^{\infty}(\Sph)})
	\norm{ e_N}_{W^{1,\infty}(\Sph)}^2
	\leq
	C( e^{-c N} + N^{-4})
	N^{4}
	\norm{ e_N}_{L^{2}(\Sph)}^2
	\leq
	C
	\norm{ e_N}_{L^{2}(\Sph)}^2 , \notag
\end{align}
for some suitable generic constants $C$ and $c$. The other terms on the right-hand side of \eqref{eq:pdf-final-error-eq1} can be handled similarly.
Upon absorption, we get
\begin{align}\label{eq:pdf-final-error-eq2}
	&\frac{\d}{\d t}\norm{e_N}_{L^2(\Gamma_{N,*}(t))}^2
	+
	\norm{\nabla_{\Sph} e_N\cdot n\circ X}_{L^2(\Sph)}^2
	\leq C e^{-c N} +
	C \norm{ e_N}_{L^{2}(\Sph)}^2
	.
\end{align}
With $T'$ defined in \eqref{eq:ind-hypo}, we integrate \eqref{eq:pdf-final-error-eq2} over $[0,T']$, use norm equivalence between $\Sph$ and $\Gamma_{N,*}(t)$, and apply Gronwall's inequality to obtain
	\begin{equation*}
	  \norm{e_N}_{L_t^\infty L_x^2([0,T']\times\Sph)}
	  +
	  \norm{\nabla_{\Sph} e_N\cdot n\circ X}_{L_t^2L_x^2([0,T']\times\Sph)}
	  \leq C e^{-c N}.
	\end{equation*}
The constants here depend on $X$ and $T'$, but not on
$N$.
By choosing sufficiently large $N$ and using the inverse inequality, we can achieve
\begin{align}
	\norm{e_N}_{L_t^\infty W_x^{1,\infty}([0,T']\times \Sph)}
	\leq 
	\frac{1}{2}N^{-4} < N^{-4}. \notag
\end{align}
Since the error equation \eqref{eq:error-eq} is a finite-dimensional ODE system, there exists $\Delta T>0$ such that $\norm{e_N}_{L_t^\infty W_x^{1,\infty}([0,T'+\Delta T]\times \Sph)}
\leq N^{-4}$. According to the definition of $T'$, we see that $T'=T$. After using the triangle inequality, the identity $X_N-X=e_N-(1-\Lhyp_N)X$ and the approximation property of $\Lhyp_N$, the proof of Theorem \ref{thm:main} is now complete.

%

\section{Numerical experiments}

We test the semi-discrete Dziuk-type pseudo-spectral method for mean curvature flow \eqref{eq:pdf-semi-scheme} on a smooth dumbbell surface. The initial surface is prescribed as a radial graph over the unit sphere, \[ X_0(\theta,\varphi) = \rho(\theta,\varphi)\,\hat x(\theta,\varphi), \qquad \rho(\theta,\varphi) = 1+Y_{20}(\theta,\varphi), \] where \(Y_{20}\) denotes the spherical harmonic of degree \(\ell=2\) and order \(m=0\), and \(\hat x\) is the outward unit normal vector on the unit sphere. Following \cite{GS2002,VRBZ11,BV2026}, we employ the product quadrature \(\Q_R\), consisting of the \((2R+1)\)-point trapezoidal rule in the azimuthal direction and the \((\lfloor R/2\rfloor+1)\)-point Gauss quadrature in the polar direction. For this specific choice of \(\Q_R\), the associated hyperinterpolation operator \(\Lhyp_N\) can be computed by a fast Fourier--Legendre transform \cite{Mohlenkamp1999} with complexity $\mathcal O\bigl(N^2(\log N)^2+N^2\log N\bigr)$, provided that \(R=\mathcal O(N)\). At each time step, we solve the following linearly implicit Dziuk-type scheme with a sufficiently small time step size \(\tau\): 
For $m=0,1,...,M-1$ with $M:=\lfloor \frac{T}{\tau} \rfloor$,
find \(X_N^{m+1}\in[\V_N]^3\) such that \[ \left( \frac{X_N^{m+1}-X_N^m}{\tau},\phi_N \right)_{\Gamma_N^m,\Q_R} + \left( \nabla_{\Gamma_N^m}X_N^{m+1}, \nabla_{\Gamma_N^m}\phi_N \right)_{\Gamma_N^m,\Q_R} =0 \] for all test functions \(\phi_N\in[\V_N]^3\). Here all metric quantities and quadrature weights are evaluated on the previous surface \(\Gamma_N^m\). In the experiments, we set the quadrature order to \(R=3N\) and the final time to \(T=0.1\). The spherical harmonic degree is chosen as $N=4,6,8,10,12,14.$ Since no exact solution is available for this dumbbell evolution, the errors are measured against an over-resolved reference solution with \(N_{\rm ref}=18\).

For \(*\in\{L^2,H^1,L^\infty\}\), we define the relative errors for the position and velocity at the final time $T$ by
\[
E_X^*(N)
=
\frac{
	\|X_N(T)-X_{N_{\rm ref}}(T)\|_*
}{
	\|X_{N_{\rm ref}}(T)\|_*
},
\qquad
E_V^*(N)
=
\frac{
	\|V_N(T)-V_{N_{\rm ref}}(T)\|_*
}{
	\|V_{N_{\rm ref}}(T)\|_*
},
\]
where $V_N:=\partial_t X_N$.
The resulting errors are depicted in
Figure~\ref{fig:dumbbell-mcf-semilogy}. We observe that both the position and velocity errors exhibit spectral convergence, which is consistent with the spectral convergence predicted by the main theorem (Theorem \ref{thm:main}).
\begin{figure}[t]
	\centering
	\begin{minipage}[t]{0.48\textwidth}
		\centering
		\includegraphics[width=\linewidth]{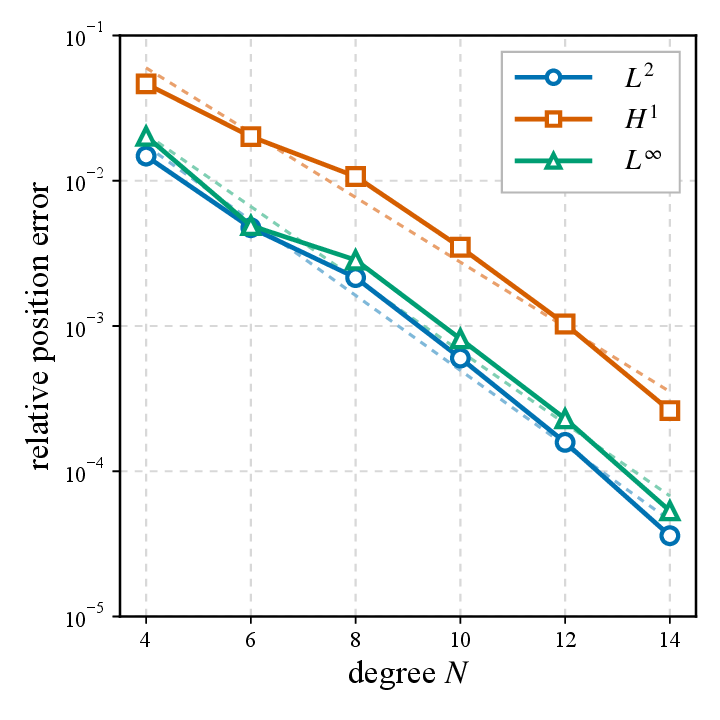}
		\centerline{(a) Position errors}
	\end{minipage}
	\hfill
	\begin{minipage}[t]{0.48\textwidth}
		\centering
		\includegraphics[width=\linewidth]{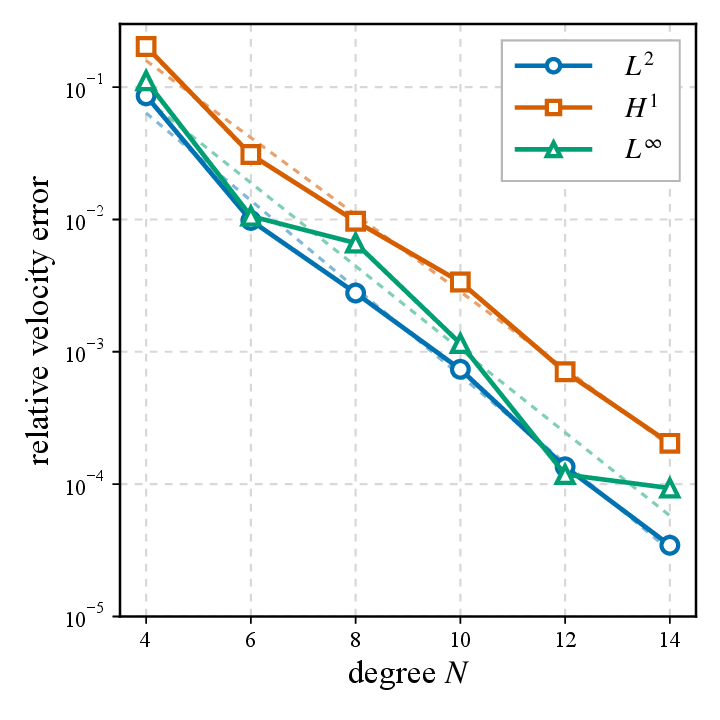}
		\centerline{(b) Velocity errors}
	\end{minipage}
	\caption{
		Spatial convergence for the dumbbell experiment at \(T=0.1\) with
		\(\tau=10^{-4}\).  The reference solution uses \(N_{\rm ref}=18\).
		Left: relative position errors in the \(L^2\), \(H^1\), and \(L^\infty\) norms.
		Right: relative velocity errors in the \(L^2\), \(H^1\), and \(L^\infty\) norms.
		The nearly linear decay on the semilogarithmic scale is consistent with
		spectral convergence in the spherical harmonic degree \(N\).
	}
	\label{fig:dumbbell-mcf-semilogy}
\end{figure}

\renewcommand{\refname}{\bf References}
\bibliographystyle{abbrv}
\bibliography{spec_MCF}

@article{ZB2025,
	title={Spherical harmonics-based pseudo-spectral method for quantitative analysis of symmetry breaking in wrinkling of shells with soft cores},
	author={Zavodnik, Jan and Brojan, Miha},
	journal={Comput. Methods Appl. Mech. Eng.},
	volume={433},
	pages={117529},
	year={2025},
	publisher={Elsevier}
}

@article{FBS2023,
	title={A spectral boundary integral method for simulating electrohydrodynamic flows in viscous drops},
	author={Firouznia, Mohammadhossein and Bryngelson, Spencer H and Saintillan, David},
	journal={J. Comput. Phys.},
	volume={489},
	pages={112248},
	year={2023},
	publisher={Elsevier}
}

@article{SW2000,
	title={Constructive polynomial approximation on the sphere},
	author={Sloan, Ian H. and Womersley, Robert S.},
	journal={J. Approx. Theory},
	volume={103},
	number={1},
	pages={91--118},
	year={2000},
	publisher={Elsevier}
}

@article{GLS2011,
	title={A pseudospectral quadrature method for {Navier--Stokes} equations on rotating spheres},
	author={Ganesh, Mahadevan and Le Gia, Quoc T. and Sloan, Ian H.},
	journal={Math. Comp.},
	volume={80},
	number={275},
	pages={1397--1430},
	year={2011}
}

@article{GS2002,
	title={Fully discrete spectral boundary integral methods for {Helmholtz} problems on smooth closed surfaces in {$\mathbb{R}^{3}$}},
	author={Graham, Ivan G. and Sloan, Ian H.},
	journal={Numer. Math.},
	volume={92},
	number={2},
	pages={289--323},
	year={2002},
	publisher={Springer}
}

@article{Sloan1995,
	title={Polynomial interpolation and hyperinterpolation over general regions},
	author={Sloan, Ian H},
	journal={J. Approx. Theory},
	volume={83},
	number={2},
	pages={238--254},
	year={1995},
	publisher={Elsevier}
}

@book{Trefethen2000book,
	title={Spectral Methods in {MATLAB}},
	author={Trefethen, Lloyd N.},
	year={2000},
	publisher={SIAM}
}

@book{AH2012book,
	title={Spherical Harmonics and Approximations on the Unit Sphere: An Introduction},
	author={Atkinson, Kendall E. and Han, Weimin},
	volume={2044},
	year={2012},
	publisher={Springer Science \& Business Media}
}

@incollection{Xu2014,
	title={Best polynomial approximation on the unit sphere and the unit ball},
	author={Xu, Yuan},
	booktitle={Approximation Theory XIV: San Antonio 2013},
	pages={357--375},
	year={2014},
	publisher={Springer}
}

@book{CHQZ1988book,
	title = {Spectral Methods in Fluid Dynamics},
	ISBN = {9783642841088},
	url = {http://dx.doi.org/10.1007/978-3-642-84108-8},
	DOI = {10.1007/978-3-642-84108-8},
	publisher = {Springer Berlin Heidelberg},
	author = {Canuto,  Claudio and Hussaini,  M. Yousuff and Quarteroni,  Alfio and Zang,  Thomas A.},
	year = {1988}
}

@book{DX2013book,
	title={Approximation Theory and Harmonic Analysis on Spheres and Balls},
	author={Dai, Feng and Xu, Yuan},
	year={2013},
	publisher={Springer}
}

@book{KP2002book,
	title={A Primer of Real Analytic Functions},
	author={Krantz, Steven G and Parks, Harold R},
	year={2002},
	publisher={Springer Science \& Business Media}
}

@article{Seeley1969,
	title={Eigenfunction expansions of analytic functions},
	author={Seeley, R. T.},
	journal={Proc. Amer. Math. Soc.},
	volume={21},
	number={3},
	pages={734--738},
	year={1969},
	publisher={JSTOR}
}

@book{Roman2005book,
	title={Advanced Linear Algebra},
	author={Roman, Steven},
	year={2005},
	publisher={Springer}
}

@article{LN2005,
	title={A posteriori error analysis for the mean curvature flow of graphs},
	author={Lakkis, Omar and Nochetto, Ricardo H},
	journal={SIAM J. Numer. Anal.},
	volume={42},
	number={5},
	pages={1875--1898},
	year={2005},
	publisher={SIAM}
}

@phdthesis{Fritz_thesis,
	title={Finite Elemente Approximation der {R}icci-{K}rümmung und Simulation des {R}icci-{D}eTurck-{F}lusses},
	author={Fritz, Hans},
	school={Albert-Ludwigs-Universität Freiburg},
	year={2013}
}

@book{Eck12,
  title = {Regularity Theory for Mean Curvature Flow},
  author = {K. Ecker},
  series = {},
  publisher = {{Springer}},
  location = {},
  year = {2012}
}

@article{CY2007,
	title={Uniqueness and pseudolocality theorems of the mean curvature flow},
	author={Chen, Bing-Long and Yin, Le},
	journal={Comm. Anal. Geom.},
	volume = {15},
	number = {},
	pages = {435--490},
	year={2007}
}

@article{EH1989,
	title={Mean curvature evolution of entire graphs},
	author={Ecker, Klaus and Huisken, Gerhard},
	journal={Ann. of Math.},
	volume={130},
	number={3},
	pages={453--471},
	year={1989},
	publisher={JSTOR}
}

@article{EH1991,
	title={Interior estimates for hypersurfaces moving by mean curvature},
	author={Ecker, Klaus and Huisken, Gerhard},
	journal={Invent. Math.},
	volume={105},
	number={1},
	pages={547--569},
	year={1991}
}

@article{Huisken1984,
	title={Flow by mean curvature of convex surfaces into spheres},
	author={Huisken, Gerhard},
	journal={J. Differential Geom.},
	volume={20},
	number={1},
	pages={237--266},
	year={1984},
	publisher={Lehigh University}
}

@article{ST2018,
	title={A highly accurate boundary integral equation method for surfactant-laden drops in {3D}},
	author={Sorgentone, Chiara and Tornberg, Anna-Karin},
	journal={J. Comput. Phys.},
	volume={360},
	pages={167--191},
	year={2018},
	publisher={Elsevier}
}

@article{Huisken1990,
	title={Asymptotic behavior for singularities of the mean curvature flow},
	author={Huisken, Gerhard},
	journal={J. Differential Geom.},
	volume={31},
	number={1},
	pages={285--299},
	year={1990},
	publisher={Lehigh University}
}

@article{Ecker2001,
	title={A local monotonicity formula for mean curvature flow},
	author={Ecker, Klaus},
	journal={Ann. of Math.},
	volume={154},
	number={2},
	pages={503--525},
	year={2001},
	publisher={JSTOR}
}

@article{White2005,
	title={A local regularity theorem for mean curvature flow},
	author={White, Brian},
	journal={Ann. of Math.},
	volume={161},
	number={3},
	pages={1487--1519},
	year={2005},
	publisher={JSTOR}
}

@article{DeTurck1983,
	title={Deforming metrics in the direction of their {R}icci tensors},
	author={DeTurck, Dennis M},
	journal={J. Differential Geom.},
	volume={18},
	number={1},
	pages={157--162},
	year={1983},
	publisher={Lehigh University}
}

@book{Man11,
  title = {{Lecture Notes on Mean Curvature Flow}},
  author = {Carlo Mantegazza},
  series = {Applied {{Mathematical Sciences}}},
  publisher = {{Springer Basel AG}},
  location = {{Basel, Switzerland}},
  year = {2012}
}

@incollection{BGN2020,
	AUTHOR={Barrett, J.W. and Garcke, H. and N{\"u}rnberg, R.},
	TITLE={Parametric finite element approximations of curvature driven interface evolutions},
	booktitle={Handb. Numer. Anal.},
	volume={21},
	pages={275--423},
	year={2020},
	publisher={Elsevier}
}

@unpublished{BV2026,
	title={A structure-preserving fast spectral method for locally inextensible vesicles with tangential smoothing},
	author={Bai, Genming and Veerapaneni, Shravan},
	note={To be submitted},
	year={2026}
}

@article{Mohlenkamp1999,
	title={A fast transform for spherical harmonics},
	author={Mohlenkamp, Martin J.},
	journal={J. Fourier Anal. Appl.},
	volume={5},
	number={2},
	pages={159--184},
	year={1999},
	publisher={Springer}
}

@article{KLL-Willmore,
  title = {{A convergent evolving finite element algorithm for Willmore flow of closed surfaces}},
  author = {Bal\'azs Kov\'acs and Buyang Li and Christian Lubich},
  volume = {149},
  number = {},
  pages = {595--643},
  journal = {Numer. Math.},
  year = {2021}
  }

@article{BL22A,
    author = {Genming Bai and Buyang Li},
    title = "{Erratum: Convergence of Dziuk's semidiscrete finite element method for mean curvature flow of closed surfaces with high-order finite elements}",
    journal = {SIAM J. Numer. Anal.},
    volume = {61},
    number = {3},
    pages = {1609--1612},
    year = {2023},
}

@article{BL2024,
	title={A new approach to the analysis of parametric finite element approximations to mean curvature flow},
	author={Bai, Genming and Li, Buyang},
	journal={Found. Comput. Math.},
	volume={24},
	number={5},
	pages={1673--1737},
	year={2024},
	publisher={Springer}
}

@article{BGV2026,
	author = {Genming Bai and Harald Garcke and Shravan Veerapeneni},
	title = {Convergence analysis for the {Barrett--Garcke--N{\"u}rnberg} method of transport type},
	journal = {Numer. Math.},
	volume = {158},
	year = {2026},
	pages = {361--410},
}

@article{BL2025,
	title={Convergence of a stabilized parametric finite element method of the {B}arrett--{G}arcke--{N}{\"u}rnberg type for curve shortening flow},
	author={Bai, Genming and Li, Buyang},
	journal={Math. Comp.},
	volume={94},
	number={355},
	pages={2151--2220},
	year={2025}
}

@article{VGZB09,
	title={A boundary integral method for simulating the dynamics of inextensible vesicles suspended in a viscous fluid in {2D}},
	author={Veerapaneni, Shravan K and Gueyffier, Denis and Zorin, Denis and Biros, George},
	journal={J. Comput. Phys.},
	volume={228},
	number={7},
	pages={2334--2353},
	year={2009},
	publisher={Elsevier}
}

@article{DFT2016,
	title={Reverse {H{\"o}lder} inequality for spherical harmonics},
	author={Dai, Feng and Feng, Han and Tikhonov, Sergey},
	journal={Proc. Amer. Math. Soc.},
	volume={144},
	number={3},
	pages={1041--1051},
	year={2016}
}

@article{DGT2020,
	title={Nikolskii constants for polynomials on the unit sphere},
	author={Dai, Feng and Gorbachev, Dmitry and Tikhonov, Sergey},
	journal={J. Anal. Math.},
	volume={140},
	number={1},
	pages={161--185},
	year={2020},
	publisher={Springer}
}

@article{LM2006,
	title={Polynomial operators and local approximation of solutions of pseudo-differential equations on the sphere},
	author={Le Gia, Quoc Thong and Mhaskar, Hrushikesh N},
	journal={Numer. Math.},
	volume={103},
	number={2},
	pages={299--322},
	year={2006},
	publisher={Springer}
}

@article{VGBZ09,
	title={A numerical method for simulating the dynamics of {3D} axisymmetric vesicles suspended in viscous flows},
	author={Veerapaneni, Shravan K and Gueyffier, Denis and Biros, George and Zorin, Denis},
	journal={J. Comput. Phys.},
	volume={228},
	number={19},
	pages={7233--7249},
	year={2009},
	publisher={Elsevier}
}

@article{VRBZ11,
	title={A fast algorithm for simulating vesicle flows in three dimensions},
	author={Veerapaneni, Shravan K and Rahimian, Abtin and Biros, George and Zorin, Denis},
	journal={J. Comput. Phys.},
	volume={230},
	number={14},
	pages={5610--5634},
	year={2011},
	publisher={Elsevier}
}

@Article{DD2006,
  Title                    = {{Error analysis of a finite element method for the Willmore flow of graphs}},
  Author                   = {Klaus Deckelnick and Gerhard Dziuk},
  volume={8},
  number={},
  pages={21--46},
  year={2006},
  journal={Interfaces Free Bound.},
}

@article{BMN2004,
	title={Surface diffusion of graphs: variational formulation, error analysis, and simulation},
	author={B{\"a}nsch, Eberhard and Morin, Pedro and Nochetto, Ricardo H},
	journal={SIAM J. Numer. Anal.},
	volume={42},
	number={2},
	pages={773--799},
	year={2004},
	publisher={SIAM}
}

@article{DD1995,
	title={Convergence of a finite element method for non-parametric mean curvature flow},
	author={Deckelnick, Klaus and Dziuk, Gerhard},
	journal={Numer. Math.},
	volume={72},
	number={2},
	pages={197--222},
	year={1995},
	publisher={Springer}
}

@article{DD2000,
	title={Error estimates for a semi-implicit fully discrete finite element scheme for the mean curvature flow of graphs},
	author={Deckelnick, Klaus and Dziuk, Gerhard},
	journal={Interfaces Free Bound.},
	volume={2},
	number={4},
	pages={341--359},
	year={2000}
}

@book{STW2011book,
	title={Spectral Methods: Algorithms, Analysis and Applications},
	author={Shen, Jie and Tang, Tao and Wang, Li-Lian},
	volume={41},
	year={2011},
	publisher={Springer Science \& Business Media}
}

@article {KLL19,
    author = {Bal\'azs Kov{\'a}cs and Buyang Li and Christian Lubich},
     title = {{A convergent evolving finite element algorithm for mean curvature flow of closed surfaces}},
   journal = {Numer. Math.},
    volume = {143},
      year = {2019},
    number = {},
     pages = {797--853},
 }

@article{BGN2007JCP,
  title = {A Parametric Finite Element Method for Fourth Order Geometric Evolution Equations},
  author = {John W. Barrett and Harald Garcke and Robert N\"urnberg},
  volume = {222},
  number = {},
  pages = {441--467},
  journal = {J. Comput. Phys.},
  year = {2007}
}

@article{BGN2008JCP,
  author = {John W. Barrett and Harald Garcke and Robert N\"urnberg},
  date = {2008-04},
  journaltitle = {J. Comput. Phys.},
  volume = {227},
  number = {},
  pages = {4281--4307},
  langid = {english},
  title = {On the parametric finite element approximation of evolving hypersurfaces in {$\mathbb{R}^3$}},
  journal = {J. Comput. Phys.},
  year = {2008}
}

@article{DDE2005,
  title = {Computation of Geometric Partial Differential Equations and Mean Curvature Flow},
  author = {Deckelnick, Klaus and Dziuk, Gerhard and Elliott, Charles M.},
  date = {2005-05},
  journaltitle = {Acta Numer.},
  volume = {14},
  pages = {139--232},
  publisher = {{Cambridge University Press}},
  langid = {english},
  journal = {Acta Numer.},
  year = {2005}
}

@article{DE2013,
  title = {Finite element methods for surface {PDEs}},
  author = {Gerhard Dziuk and Charles M. Elliott},
  date = {2013-05},
  journaltitle = {Acta Numer.},
  volume = {22},
  pages = {289--396},
  publisher = {{Cambridge University Press}},
  langid = {english},
  journal = {Acta Numer.},
  year = {2013}
}

@article{EF2017,
  title = {On approximations of the curve shortening flow and of the mean curvature flow based on the {DeTurck} trick},
  author = {Charles M. Elliott and Hans Fritz},
  date = {2017},
  journaltitle = {IMA J. Numer. Anal.},
  volume = {37},
  number = {2},
  pages = {543--603},
  langid = {english},
  journal = {IMA J. Numer. Anal.},
  year = {2017}
}

@article{Li21,
  title = {{Convergence of Dziuk's semidiscrete finite element method for mean curvature flow of closed surfaces with high-order finite elements}},
  author = {Li, Buyang},
  pages = {1592--1617},
  volume = {59},
  journal = {SIAM J. Numer. Anal.},
  year = {2021},
}

\end{document}